\documentclass[11pt, leqno]{amsart}

\usepackage[square, numbers, comma]{natbib}

\usepackage[usenames,dvipsnames,svgnames,table]{xcolor}
\usepackage{amsmath,amsthm,amssymb,amssymb,esint,verbatim,tabularx,graphicx}
\usepackage{fancyhdr}
\usepackage{enumerate}
\usepackage{amsmath}
\usepackage{fancyhdr}
\usepackage{epic}
\usepackage{pgf,tikz}
\usetikzlibrary{arrows}

\usepackage[utf8]{inputenc}
\usepackage{color}
\usepackage{hyperref}
\usepackage{pdfpages}
\usepackage{mathtools}

\hypersetup{urlcolor=blue, colorlinks=true}
\usepackage[inner=3cm,outer=3cm,bottom=4cm,top=4cm]{geometry}

\newtheorem{theorem}{Theorem}[section]

\newtheorem{definition}[theorem]{Definition}
\newtheorem{lemma}[theorem]{Lemma}
\newtheorem{remark}[theorem]{Remark}
\newtheorem{corollary}[theorem]{Corollary}

\newcommand{\R}{\ensuremath{\mathbb{R}}}
\newcommand{\N}{\ensuremath{\mathbb{N}}}

\newcommand{\Z}{\ensuremath{\mathbb{Z}}}
\newcommand{\Levy}{\ensuremath{\mathcal{L}}}

\newcommand{\Operator}{\ensuremath{\mathfrak{L}^{\sigma,\mu}}}

\newcommand{\Levymu}{\ensuremath{\mathcal{L}^\mu}}

\newcommand{\veps}{\varepsilon}

\newcommand{\FL}{(-\Delta)^{s}}

\newcommand{\vv}{\hat{v}}

\newcommand{\dd}{\,\mathrm{d}}

\newcommand{\dell}{\partial}

\newcommand{\indik}{\mathbf{1}}

\newcommand{\e}{\text{e}}

\DeclareMathOperator*{\esssup}{ess \, sup}

\DeclareMathOperator*{\esslim}{ess\,lim}

\DeclareMathOperator{\supp}{supp}

\numberwithin{equation}{section}

\begin{document}

%\vspace{-1.2cm}

\title[The one-phase fractional  Stefan problem]{The one-phase fractional  Stefan problem}

\author[F.~del Teso]{F\'elix del Teso}
\address[F.~del Teso]{Departamento de An\'alisis Matem\'atico y Matem\'atica Aplicada\\
Universidad Complutense de Madrid\\
28040 Madrid, Spain}
\email[]{fdelteso\@@{}ucm.es}
\urladdr{https://sites.google.com/view/felixdelteso}

\author[J.~Endal]{J\o rgen Endal}
\address[J. Endal]{Departamento de Matem\'aticas\\
Universidad Aut\'onoma de Madrid (UAM)\\
28049 Madrid, Spain}
\email[]{jorgen.endal\@@{}uam.es}
\urladdr{http://verso.mat.uam.es/~jorgen.endal/}

\author[J.~L.~V\'azquez]{Juan Luis V\'azquez}
\address[J.~L.~V\'azquez]{Departamento de Matem\'aticas\\
Universidad Aut\'onoma de Madrid (UAM)\\
28049 Madrid, Spain}
\email[]{juanluis.vazquez\@@{}uam.es}
\urladdr{http://verso.mat.uam.es/~juanluis.vazquez/}

\keywords{Stefan problem, phase transition, long-range interactions, free boundaries, nonlinear and nonlocal equation, fractional diffusion.}

\subjclass[2010]{
80A22, %Stefan problems, phase changes
35D30, %Weak solutions
35K15, %Initial value problems for second-order parabolic equations
35K65, %Degenerate parabolic equations
35R09, % 	Integro-partial differential equations
35R11, %Fractional partial differential equations
%45K05, %Integro-partial differential equations
65M06, %Finite difference methods
65M12%Stability and convergence of numerical methods
%76S05, %Flows in porous media; filtration; seepage
}

\begin{abstract}
We study the existence and properties of solutions and free boundaries of the one-phase Stefan problem with fractional diffusion posed  in $\R^N$.  In terms of the enthalpy $h(x,t)$, the evolution equation reads $\dell_t h+ \FL  \Phi(h) =0$, while the temperature is defined as $u:=\Phi(h):=\max\{h-L,0\}$ for some constant $L>0$  called the latent heat, and $\FL$ stands for the fractional Laplacian with exponent $s\in(0,1)$.

We prove the existence of a continuous and bounded selfsimilar solution of the form $h(x,t)=H(x\,t^{-1/(2s)})$ which exhibits a free boundary at the change-of-phase level $h(x,t)=L$. This level is located at the line (called the free boundary) $x(t)=\xi_0 t^{1/(2s)}$ for some $\xi_0>0$. The construction is done in 1D, and its extension to $N$-dimensional space is shown.

We also provide well-posedness and basic properties of very weak solutions for general bounded  data $h_0$ in several dimensions. The temperatures $u$ of these solutions are continuous functions that have finite speed of propagation, with possible free boundaries.  We obtain estimates on the growth in time of the support of $u$ for solutions with compactly supported initial temperatures.
Besides, we show the property of conservation of positivity for $u$ so that the support never recedes. On the contrary, the enthalpy $h$ has infinite speed of propagation and we obtain precise estimates on the tail.

The limits $L\to0^+$, $L\to +\infty$,  $s\to0^+$ and $s\to 1^-$ are also explored, and we find interesting connections with well-studied diffusion problems. Finally, we  propose convergent monotone finite-difference schemes and  include numerical experiments aimed at illustrating some of the obtained theoretical results, as well as other interesting phenomena.
\end{abstract}

\maketitle

% \vspace{-1.555cm}
\small
\setcounter{tocdepth}{1}
\tableofcontents
\normalsize

%%%%%%%%%%%%%%%%%%%%%%%%%%%%%%%%%%%%%%%%%%%%%%%%%%%%%%%%%%%NEW SECTION%%%%%%%%%%%%%%%%%
%%%%%%%%%%%%%%%%%%%%%%%%%%%%%%%%%%%%%%%%%%

\section{Introduction}
\label{sec.intro}

In this paper we study existence and properties of solutions of the one-phase Stefan problem with fractional diffusion posed in $\R^N$. In terms of the enthalpy variable, we look for a  bounded  function $h=h(x,t)$ which solves the following Fractional Stefan Problem (FSP):
\begin{align}
\dell_t h+ \FL  \Phi(h) =0 \qquad\qquad&\text{in}\qquad Q_T:=\R^N\times(0,T), \label{P1}\tag{P}\\
h(\cdot,0)=h_0 \qquad\qquad&\text{on}\qquad \R^N,\label{P1-Init}\tag{IC}
\end{align}
where $\Phi(h(x,t)):=(h(x,t)-L)_+=\max\{h(x,t)-L,0\}=:u(x,t)$ is the temperature, $L>0$  is constant, $N\geq1$, and $T>0$ is arbitrary. The  nonnegative operator $\FL$ denotes the fractional Laplacian acting on the space variables, with  parameter $s\in(0,1)$,  defined by
$$
\FL\psi(x):=c_{N,s}\textup{P.V.}\int_{\R^N\setminus\{x\}}\frac{\psi(x)-\psi(y)}{|x-y|^{N+2s}}\dd y
$$
for some positive normalizing constant $c_{N,s}>0$, cf. \cite{Lan72,Ste70}.  Sometimes, it is convenient to write Problem \eqref{P1} in terms of both variables as
\begin{equation}\label{P3}\tag{P'}
\dell_t h+ \FL u=0,  \qquad\qquad\text{in}\qquad\qquad Q_T=\R^N\times(0,T).
\end{equation}
Then we can talk about a solution pair $(h,u)$, where the enthalpy $h$ and the temperature $u$ are always related through the law $u=\Phi(h)=(h-L)_+$. While temperature is a rather standard physical concept, enthalpy is an energy which includes the so-called latent heat at the phase change front, here represented by the constant $L>0$. The constitutive relation between both quantities, enthalpy and temperature, is typical for the Stefan problem and similar problems of change of phase: we have given a very careful derivation based on physical grounds in \cite[Section 2]{DTEnVa20}, see also the monographs \cite{Rub71, Gup18}. This physical understanding is convenient, though not strictly necessary to follow the mathematics of the present paper.

The main feature of the paper is the existence of a {\em free boundary}, separating the region in space-time where $u>0$, i.e., $h>L$, from the space-time region where $u=0$ and $h<L$. The name {\em interface} is also used.
Actually, the constant $L>0$,  called the \emph{latent heat} in the Stefan literature,
makes the equation very degenerate for $h<L$; also, inverting $\Phi$ we may write  $h=\beta(u)$ where $\beta $ is a  multivalued graph at $u=0$. These facts are quite important in the present study. Indeed, it clearly separates this work from the studies on the fractional porous medium equation $\dell_t h+ \FL  h^m=0$ with $m>1$, where the nonlinearity is a strictly increasing function, cf. \cite{DPQuRoVa11, DPQuRoVa12, DPQuRoVa17}.

We recall that when $s=1$ we recover the (One-Phase) Classical Stefan Problem (CSP for short) which has been thoroughly studied, cf. the papers \cite{Kam61, Fri68, Duv73, CaFr79, CaEv83, Vui93, AnCaSa96, QuVa01} and the books \cite{Mei92, Rub71, Fri82, Gup18} among the extensive literature. In that sense we may say that  \eqref{P1} is a natural candidate to be the {\em One-Phase Fractional  Stefan Problem}. In contrast with the CSP, our fractional model involves long-range interactions of  fractional Laplace type, which represent nonlocal diffusions  related to L\'evy flights in the probabilistic description.

In the classical problem $s=1$, the level $h=L$ is called the phase-change level, and we will keep that denomination for our problem. It is usually said in the CSP studies that $\{u>0\}$ denotes the water phase (or hot region) and $\{h<L\}$ the ice phase (or cold region), and that language will also be kept here for intuition. As we have said, the separation between the two regions is called the free boundary. We would like to prove existence of this object and establish its mathematical properties. In particular, proving that the free boundary  is a surface in space-time is  a difficult  part of the problem.

\subsection*{Outline of results}
Let us now give an idea about the contents: After establishing existence, uniqueness  and temperature continuity of a suitable class of very weak solutions (cf. Section \ref{sec.defwp} and Appendix \ref{sec:WellPosednessInLInf}), we want to examine the
relevant properties of these solutions. One of the main features of the CSP is the property of \emph{finite propagation}, whereby compactly supported temperature initial data $u_0:=(h_0-L)_+$  produce temperature solutions with the same property for any positive time. In formulas:
\begin{equation*}
\mbox{supp}\,\{u_0\}\subset B_R(0) \implies \mbox{supp} \{u(\cdot,t)\}\,\subset B_{R+R_1(t)}(0),
\end{equation*}
for a finite (and increasing) function $R_1$ (even more, here $R_1(t)$ will be globally bounded in $t$, but that requirement is not essential and depends on the initial setting). This gives rise to the existence of a free boundary or interface where ice and water separate. We establish the finite propagation property and existence of a free boundary for our problem FSP in terms of the temperature variable $u$ in Section \ref{sec.fpp}. This property is shared by other diffusion problems involving local operators, like the Porous Medium Equation $\dell_t h- \Delta h^m =0$ , $m>1$ \cite{Vaz07}, but it does not hold for the Fractional Porous Medium $\dell_t h+ \FL  h^m =0$ for any $m>0$, nor its generalizations, see e.g. \cite{DPQuRoVa12, DPQuRoVa17}, due to the nonlocality of the fractional operator.

A preliminary tool in the proof of finite propagation is the construction of a special 1-D selfsimilar solution (SSS) of the form $h(x,t)=H(xt^{-1/(2s)})$ exhibiting a change of phase and its corresponding free boundary. This
important solution is thoroughly studied in Section \ref{sec.sss}. We show that the {\em profile function} $H$ is continuous and bounded and keeps the initial $L^\infty $ bounds in time.  The free boundary of $u(x,t)$   takes the form of the  curve
\begin{equation*}
x(t)=\xi_0 t^{\frac{1}{2s}} \qquad\textup{for all}\qquad t\in(0,T),
\end{equation*}
and some finite $\xi_0>0$. Therefore, we obtain superdiffusions for all $0<s<1$, since the time exponent is $\gamma:=1/(2s)>1/2$. Note that $\gamma=1/2$ is the Brownian scaling of the heat equation and the CSP.
The case of linear interface propagation happens for $s=1/2$, a case that turns out to be critical for some of the results.
 It is worth  to remark  that $H$ turns out to be continuous in the whole space, and thus, it does not exhibit a discontinuity at the interface point $\xi_0$. This is a surprising property that is not shared with any other local or nonlocal Stefan problem previously studied in the literature. Generally, even in the nonlocal case, one can only prove that $U:=(H-L)_+$ is continuous (as shown in \cite{AtCa10}).

Moreover, in Section \ref{sec.fpp} we will also prove infinite propagation for the enthalpy $h$ of the FSP \eqref{P1}--\eqref{P1-Init}, a property which is false for the CSP. Finally, in Section \ref{sec.conspos} we will prove conservation of the positivity region of the temperature, an important fact in the qualitative analysis of the solutions that ensures that the free boundary of a general solution does not shrink in time.

Important limit cases are examined in Section \ref{sect.limit}. Indeed, the temperature of the solution of our FSP is, on one hand, bounded above by the solutions of the linear fractional heat equation posed in the whole space with the same initial temperature data, and it is well-known that for the latter infinite propagation holds. On the other hand, the temperature of our FSP is bounded below by the solutions of the Dirichlet problem for the  fractional heat equation posed  in a bounded domain contained in the support of the initial temperature. These two problems serve as comparison on both sides, and both can be obtained from the FSP by letting the latent heat $L$ go to $0^+$ or to $+\infty$ respectively. The limits $s\to0^+$ and $s\to 1^-$ are also considered in Section \ref{sect.limit-s}; in this way, we show that  the FSP connects an ODE type equation for $s=0$ with the Classical Stefan Problem for $s=1$.

In Section \ref{sect.ab}  we prove some large time asymptotic results showing cases of convergence to the SSS we have constructed. Section \ref{sec:Num} is devoted to a numerical study of the FSP. First we propose convergent monotone finite-difference schemes and later we include numerical experiments that illustrate some of the obtained theoretical results, as well as other interesting phenomena. In Section \ref{sec:openprob}, we include some relevant comments and open problems. We comment in particular on the FSP with two phases.  Finally, two Appendices contain statements and proofs of basic material or  technical results.

\noindent {\bf Some related works.}  Fractional Stefan problems are scarcely mentioned in the mathematical literature. The first reference to our model seems to be the work \cite{AtCa10} by Athanasopoulos and Caffarelli where they prove that $u=\Phi(h)\in C(\R^N\times(0,T))$ as long as $h_0\in L^\infty(\R^N)$.

Nonlocal models of Stefan-type which take into account mid-range interactions are studied in \cite{BrChQu12, ChS-G13, CaDuLiLi18, CoQuWo18}.
Of particular interest in our context is \cite{BrChQu12}, where the equation involves convolution type operators: $\dell_t h= J*u- u$ and $ u=(h-1)_+$.
Here the interaction kernel $J$ is assumed to be continuous, compactly supported, radially symmetric,
and with $\int_{\R^N} J(x)\dd x= 1$. In contrast, our fractional model involves long-range interactions of the typical fractional Laplacian type, which are related to L\'evy flights. This difference strongly affects the behaviour of solutions, so that we may say that the theories are different. We will comment in Section \ref{sec:openprob} on the differences and similarities with our problem.

In the applied literature there are a number of references to models derived from the classical Stefan problem by introducing Caputo or Riemann-Liouville fractional derivatives in time, and in some cases also in space and/or the transmission conditions at the interface. See for instance  \cite{Vol14} and references \cite{RoCaTa20, Rys20} for recent work. We are not dealing with such kind of fractional derivatives, which lead to quite different models and results.

\begin{remark}\label{rem:translationofL} \rm
There is a simple modification of the formulation that will make some calculations much easier. We explain it here to avoid possible  confusions.  Namely, if we define the  ``modified enthalpy'' as  ${\overline h}=h-L$, the relation temperature-enthalpy simplifies into $\overline{u}={\overline h}_+=\max\{\overline h, 0\}$, and the equation reads: $\dell_t \overline h+ \FL {\overline h}_+ =0$. So no generality is lost in assuming that $L=0$ and that the negative enthalpies, \ ${\overline h}<0$, define the ice region. Of course, the initial datum is also changed.

We want our solutions to be nonnegative functions or at least bounded from below. Note that any bound from below $h\ge -K$ becomes $\overline h\ge -K-L:=-L_1$. There is another simple alternative  that may be useful. By  putting  $h_1=h+K$, which is the enthalpy above the minimum level, letting $u_1=u$, and considering the problem  for $(h_1,u_1)$ with $L_1=L+K$, this is equivalent to asking for nonnegative enthalpies $h_1\ge 0,$ and the equation stays the same  with a  new nonlinearity: $u_1=(h_1-L_1)_+$.   These remarks will be used in what follows.
\end{remark}

\vskip-0.1cm

\noindent {\bf Notation  and conventions.}
For two functions $f,g$, the notation $f\lesssim g$, $f\gtrsim g$, and $f\sim g$ mean that there exist constants $C>0$ such that $f\leq Cg$, $f\geq Cg$, and $f=Cg$. The notation $f\asymp g$ then simply means that $f\lesssim g$ and $f\gtrsim g$. A modulus of continuity $\Lambda$ is such that $\Lambda(\lambda)\to0$ as $\lambda\to0^+$.  A number of functional spaces appear with usual notations. We just note that  $C_{\textup{b}}(\Omega)$ denotes the space of continuous and bounded functions defined in $\Omega$.\\
 In all the paper\, $s\in(0,1)$ is the order of the fractional Laplacian,  $N\ge 1$ is the spatial dimension (in some sections restricted to $N=1$), and the latent heat $L$ is a positive constant,
 unless we state it explicitly otherwise in some proofs using the translation trick.

\section{Definitions and well-posedness results}
\label{sec.defwp}

We will work in the framework of very weak (or distributional) solutions. The definition of distributional solutions is rather standard. It has been extensively studied in \cite{DTEnJa17a, DTEnJa17b} for a more general family of problems including \eqref{P1}--\eqref{P1-Init}.
\begin{definition}[Bounded very weak solutions]\label{def:distSolfrac}
Given $h_0\in L^\infty(\R^N)$, we say that $h\in L^\infty(Q_T)$ is a \textup{very weak} solution of \eqref{P1}--\eqref{P1-Init} if, for all $\psi\in C_{\textup{c}}^\infty (\R^N \times [0,T))$,
\begin{equation}\label{defeq:distSolFL}
\begin{split}
\int_0^T \int_{\R^N} \big(h \partial_t \psi - \Phi(h)\FL \psi\big)\dd x \dd t +\int_{\R^N} h_0(x) \psi(x,0)\dd x=0.
\end{split}
\end{equation}
\end{definition}

\begin{remark}
An equivalent alternative for \eqref{defeq:distSolFL} is $\dell_th+\FL \Phi(h)=0$ in $\mathcal{D}'(\R^N\times(0,T))$
and
\[
\esslim_{t\to0^+} \int_{\R^N} h(x,t)\psi(x,t)\dd x= \int_{\R^N} h_0(x) \psi(x,0)\dd x \quad \textup{for all} \quad \psi\in C_{\textup{c}}^\infty (\R^N \times [0,T)).
\]
\end{remark}

We  will prove the following well-posedness result of bounded very weak solutions of~\eqref{P1}--\eqref{P1-Init}, as well as the main properties:

\begin{theorem}[Well-posedness of bounded solutions]\label{thm:bdddistsol}
Given $ h_0 \in L^\infty(\R^N)$,
 there exists a unique very weak solution $h\in L^\infty(Q_T)$ of \eqref{P1}--\eqref{P1-Init} with $h_0$ as initial data. Moreover, given two very weak solutions  $h,\hat{h}\in L^\infty(Q_T)$ with initial data $h_0,\hat{h}_0 \in L^\infty(\R^N)$, they have the following properties:
\begin{enumerate}[{\rm (a)}]
\item\label{thmfraclinf2-item1} $\|h(\cdot,t)\|_{L^\infty(\R^N)}\leq \|h_0\|_{L^\infty(\R^N)}\, $ for a.e. $t\in(0,T)$.
\item\label{thmfraclinf2-item2} If $h_0\leq \hat{h}_0$ a.e. in $\R^N$, then $h \leq\hat{h}$ a.e. in $Q_T$.
\item\label{thmfraclinf2-item3}  If $(h_0-\hat{h}_0)^+\in L^1(\R^N)$, then
\[
\int_{\R^N}(h(x,t)-\hat{h}(x,t))^+\dd x\leq \int_{\R^N}(h_0(x)-\hat{h}_0(x))^+\dd x \qquad\text{for a.e.}\qquad t\in(0,T).
\]
\item\label{thmfraclinf2-item4} If $h_0\in L^1(\R^N)$, then $
\int_{\R^N} h(x,t)\dd x=\int_{\R^N} h_0(x)\dd x$ for a.e. $t\in(0,T)$.

\item\label{thmfraclinf2-item5} If $\|h_0(\cdot+\xi)-h_0\|_{L^1(\R^N)}\to 0$ as $|\xi|\to 0^+$, then $h\in C([0,T]:L^1_{\textup{loc}}(\R^N))$. Moreover, for every $t,s\in[0,T]$ and every compact set $K\subset\R^N$,
$$
\|h(\cdot,t)-h(\cdot,s)\|_{L^1(K)}\leq \Lambda(|t-s|)
$$
where $\Lambda$ is a modulus of continuity depending on $K$ and $\|h_0(\cdot+\xi)-h_0\|_{L^1(\R^N)}$.
\end{enumerate}
\end{theorem}

Uniqueness (and existence) of bounded very weak solutions of  \eqref{P1}--\eqref{P1-Init} is established in \cite{GrMuPu19}, where $\Phi$ can be any nondecreasing locally Lipschitz function.

The proof of existence and properties \eqref{thmfraclinf2-item1}--\eqref{thmfraclinf2-item5} (for a larger class of nonlocal operators) is delayed to  Appendix \ref{sec:WellPosednessInLInf}, see Theorem \ref{thm:Linfty-dist}.  We proceed in this way in order to concentrate immediately on the existence of free boundaries and propagation properties.

We will also use a basic regularity result valid for our class of solutions, taken from \cite{AtCa10}.
\begin{theorem}\label{thmfrac-cont} Under the  assumptions of Theorem \ref{thm:bdddistsol},
 the temperature $u=\Phi(h)$ is continuous, $u\in C(\R^N\times(0,T))$, with a uniform modulus of continuity  for $t\geq \tau>0$.  Additionally, if $\Phi(h_0)\in C_{\textup{b}}( \Omega)$ for some open set $\Omega\subset \R^N$, then $\Phi(h)\in C_{\textup{b}}(\Omega\times[0,T))$.
\end{theorem}

For the reader's convenience we explain at the end of Appendix \ref{sec:WellPosednessInLInf} how the result \cite{AtCa10} is applied to our solutions.

%%%%%%%%%%%%%%%%%%%%%%%%%%%%%%%%%%%%%%%%%%%%%%%%%%%%%%%%%%%%%%%%%%%%%%%%%%%%%%%%%%
\section{Selfsimilar solutions}
\label{sec.sss}

The Fractional Stefan Problem \eqref{P1} is invariant under certain space-time scalings. Choosing appropriate initial data, we will get selfsimilar solutions of the form
\begin{equation*}
h(x,t)=t^{-\alpha}H(x\,t^{-\beta}),
\end{equation*}
with suitable selfsimilarity constants $\alpha$ and $\beta$, and profile function $H(\xi)$, $\xi=x\,t^{-\beta}$. It is a well-known fact  that a good understanding of this kind of solutions will also allow us to better understand many relevant properties of general solutions for many classes of nonlinear evolution problems, see the classical book \cite{Bar96}.
In the case of the Classical Stefan Problem a prominent role in the theory is played by the existence of a selfsimilar solution in one space dimension  $N=1$ with exponents $\alpha=0$  and  $\beta=1/2$.

In the case of \eqref{P1}--\eqref{P1-Init}, it is immediate from the equation that the acceptable choice of $\beta$ is $1/(2s)$. Here, we will also accept the second choice $\alpha=0$ which means conservation of $L^\infty$ bounds.
The following result describes existence and properties of the selfsimilar profile for the fractional version. It points to the main differences and similarities between the classical and fractional setting. One of the common main properties is the existence of a finite free boundary.

\begin{theorem}[Properties of the selfsimilar profile in $\R$]\label{thm:SS-all}
Assume
$N=1$,  and $P_1,P_2>0$. Consider the initial data
\begin{equation*}
h_0(x):=\begin{cases}
L+P_1  \qquad\qquad&\textup{if}\qquad   x\leq0\\
L-P_2  \qquad\qquad&\textup{if}\qquad   x>0,
\end{cases}
\end{equation*}
and let $h\in L^\infty(Q_T)$ be the corresponding very weak solution of \eqref{P1}--\eqref{P1-Init}.  Then:
\begin{enumerate}[\rm(a)]
\item\label{thm-SS-item1}  \textup{(Profile)} $h$ and $u$ are selfsimilar with formulas
 $$
 h(x,t)=H(x t^{-\frac{1}{2s}}), \quad
 u(x,t)=U(x t^{-\frac{1}{2s}})\quad \mbox{ \textup{for all } \  $(x,t)\in \R\times(0,T)$},
 $$
\noindent where the  selfsimilar profiles $H$ and $U=(H-L)_+$ satisfy the 1-D nonlocal equation:
\begin{equation}\label{P2SS}\tag{SS1}
-\frac{1}{2s} \xi H'(\xi)+ \FL U(\xi)=0 \qquad \textup{in} \qquad  \mathcal{D}'(\R).
\end{equation}
\item\label{thm-SS-item2}  \textup{(Boundedness and limits)} $L-P_2\leq H(\xi)\leq L+P_1$ for all $\xi\in\R$ and
\[
\displaystyle\lim_{\xi\to-\infty}H(\xi)=L+P_1 \qquad  \textup{and} \qquad \displaystyle\lim_{\xi\to+\infty}H(\xi)=L-P_2.
\]
\item\label{thm-SS-item3} \textup{(Free boundary)} There exists a unique finite $\xi_0>0$  such that $H(\xi_0)=L$.  This means that the free boundary of the space-time solution $h(x,t)$  at the level $L$  is given by the  curve
    \begin{equation*}
    x=\xi_0\,t^{\frac{1}{2s}} \qquad\textup{for all}\qquad t\in(0,T).
    \end{equation*}
Moreover, $\xi_0>0$ depends only on $s$ and the ratio $P_2/P_1$ (but not on $L$).

\item\label{thm-SS-item5}  \textup{(Monotonicity)} $H$ is nonincreasing. Moreover, $H$ is strictly decreasing in $[\xi_0,+\infty)$.
\item\label{thm-SS-item4} \textup{(Regularity)} $H\in C_\textup{b}(\R)$.  Moreover, $H\in C^\infty((\xi_0,+\infty))$, $H\in C^{1,\alpha}((-\infty,\xi_0))$ for some $\alpha>0$,   and \eqref{P2SS} is satisfied in the classical sense in $\R\setminus\{\xi_0\}$.

\item\label{thm-SS-item6} \textup{(Behaviour near the free boundary)} For $\xi$ close to $\xi_0$ and  $\xi\leq \xi_0$,
\[
U(\xi)=H(\xi)-L = O((\xi_0-\xi)^{s}).
\]
\item\label{thm-SS-item7}  \textup{(Fine behaviour at $+\infty$)} For all $  \xi> \xi_0$, we have $H'(\xi)<0$ and, for $\xi\gg \xi_0$
$$
H(\xi)- (L-P_2)\asymp 1/|\xi|^{2s}, \qquad H'(\xi)\asymp-1/|\xi|^{1+2s}.
$$

 \item\label{thm-SS-item8}  \textup{(Mass transfer)}
 Moreover, if $s>1/2$ then
\[
\int_{-\infty}^0\big( (L+P_1)- H(\xi)\big)\dd \xi= \int_0^\infty \big(H(\xi)-(L-P_2)\big) \dd \xi<+\infty.
\]
%If $s\leq 1/2$ both integrals above are infinite.
\end{enumerate}
\end{theorem}

\begin{remark}\label{rem:nodeponL} \rm
Actually, the value of $L$ is not relevant mathematically speaking. We could thus work with $L=0$ as discussed in Remark \ref{rem:translationofL}. A particular consequence is that  $\xi_0$ cannot depend on $L$, but just on $s$, $P_1$ and $P_2$.  Moreover,  it only depends on $s$ and the quotient $P_2/P_1$: Note that if $L=0$, then $\hat{h}:=K h$ is also a solution of \eqref{P1}--\eqref{P1-Init} with initial data $\hat{h}_0(x)=K h_0(x)$. Thus the free boundary point $\xi_0$ coincides for $h$ and $\hat{h}$. Taking $K=1/P_1$  we get that
\begin{equation*}
\hat{h}_0(x):=\begin{cases}
1   \qquad\qquad&\textup{if}\qquad   x\leq0\\
-P_2/P_1  \qquad\qquad&\textup{if}\qquad   x>0,
\end{cases}
\end{equation*}
showing the desired dependence on $P_2/P_1$.
\end{remark}

The proof of Theorem \ref{thm:SS-all} is one of the main topics of this paper, and it will be done step by step (not necessarily in the order stated),  in a series of partial results distributed from Section \ref{sec:prelimfacts} to Section \ref{sec:decayandmasstransfer}.

\subsection{Results in several dimensions}

It is easy to extend this selfsimilar solution to higher dimensions. Essentially, the multi-D selfsimilar solution is a constant extension of $H$ in the new spatial variables.

\begin{corollary}[Properties of the selfsimilar profile in $\R^N$]\label{cor:NSS-all}
Assume
$P_1,P_2>0$, $x,\xi\in \R$, and $x',\xi'\in \R^{N-1}$.  Let $h_0,h,H$ be as in Theorem \ref{thm:SS-all}, and consider the initial condition $\tilde{h}_0\in L^\infty(\R^N)$ given by $\tilde{h}_0(x,x'):=h_0(x)$, i.e.,
\begin{equation*}
\tilde{h}_0(x,x')=\begin{cases}
L+P_1  \qquad\qquad&\textup{if}\qquad   x\leq0\\
L-P_2  \qquad\qquad&\textup{if}\qquad   x>0,
\end{cases}
\end{equation*}
and let $\tilde{h}\in L^\infty(Q_T)$ be the corresponding very weak solution of \eqref{P1}--\eqref{P1-Init}. Then $\tilde{h}$ and $\tilde{u}$ are  selfsimilar and
\begin{enumerate}[{\rm (a)}]
\item \textup{(Profile)} $\tilde{h}(x,x',t)=\tilde{H}(x t^{-1/(2s)}, x't^{-1/(2s)})$ for all $(x,x',t)\in \R^N\times(0,T)$ and $\tilde{H}$ satisfies
 \begin{equation*}
-\frac{1}{2s} (\xi,\xi')\cdot \nabla \tilde{H}(\xi,\xi')+ \FL (\tilde{H}(\cdot,\cdot)-L)_+)(\xi,\xi')=0 \quad \textup{in} \quad  \mathcal{D}'(\R^N).
\end{equation*}
\item \textup{(Constant extension)} $\tilde{H}(\xi,\xi')=H(\xi)$ and  $\tilde{U}(\xi,\xi'):=U(\xi)$ for all $(\xi,\xi')\in \R^N$.
\item \textup{(Properties)} For every $\xi'\in \R^{N-1}$,  $\tilde{H}(\cdot,\xi')$ and $\tilde{U}(\cdot,\xi')$ satisfy properties \eqref{thm-SS-item1}--\eqref{thm-SS-item8} of Theorem \ref{thm:SS-all}.
\end{enumerate}
\end{corollary}

Corollary \ref{cor:NSS-all} is a trivial consequence of Theorem \ref{thm:SS-all} and the fact that $
(-\Delta_N)^{s} \tilde{U}(\xi,\xi')=(-\Delta_1)^{s} U (\xi)$ (cf. Lemma \ref{lem:secSol}) where $-\Delta_N$ and $-\Delta_1$ denote the Laplacian in $\R^N$ and in $\R$ respectively.

\begin{remark}\label{rem:hyperplanes} \rm
Since equation \eqref{P1} is translationally and rotationally invariant, we actually have a bigger family of selfsimilar solutions arising from Corollary \ref{cor:NSS-all}. Indeed, given any hyperplane $\mathcal{H}\subset \R^N$, consider open sets $\mathcal{H}^+, \mathcal{H}^-$ such that $\mathcal{H}^+\cup \mathcal{H}^-=\R^N\setminus \mathcal{H}$ and $\mathcal{H}^+\cap \mathcal{H}^-=\emptyset$. E.g., take $\mathcal{H}=\{0\}\times \R^{N-1}$, $\mathcal{H}^-=(-\infty,0)\times  \R^{N-1}$, and $\mathcal{H}^+=(0,+\infty)\times  \R^{N-1}$.  Indeed, the initial data
\begin{equation}\label{eq:HeavySide1roratedtranslated}
\tilde{h}_0(y):=\begin{cases}
L+P_1  \qquad\qquad&\textup{if}\qquad   y\in \overline{\mathcal{H}^-}\\
L-P_2  \qquad\qquad&\textup{if}\qquad   y \in \mathcal{H}^+,
\end{cases}
\end{equation}
also produces a selfsimilar solution that is a translation and rotation of the selfsimilar solution $\tilde{h}$ described in Corollary \ref{cor:NSS-all}.

\end{remark}

\subsection{Preliminary facts on selfsimilarity}\label{sec:prelimfacts}

First we will establish some general properties of the selfsimilar profile in $\R^N$.
Note that the equation has the following scaling invariance:
\begin{lemma}[Space-time selfsimilarity of the equation]\label{lem:ssofequation}
Let $h\in L^\infty(Q_T)$ be a very weak solution of \eqref{P1}--\eqref{P1-Init} with initial data $h_0\in L^\infty(\R^N)$. Then for all $a>0$ the function $h_a(x,t):=h(ax, a^{2s} t)$ is a very weak solution of \eqref{P1}--\eqref{P1-Init} with initial data $h_{0,a}(x):= h(ax)$.
\end{lemma}
\begin{proof}
Let $\psi\in C_\textup{c}^\infty(Q_T)$, and define $\phi(y,\tau):= \psi(y/a,\tau/b)$ for some $a,b>0$. Then, by changing variables $ax=y$ and $bt=\tau$ and the scaling properties of the time derivative and the fractional Laplacian (cf. Lemma \ref{lem:FLHomogeneity}), we get
\begin{equation*}
\begin{split}
\int_0^T \int_{\R^N} h(ax,bt) \dell_t\psi(x,t)\dd x\dd t
&= \frac{1}{a^N\, b} \int_0^{bT} \int_{\R^N} h(y,\tau) \dell_t\psi\left(\frac{y}{a},\frac{\tau}{b}\right)\dd y\dd \tau\\
&=\frac{1}{a^N}  \int_0^{bT} \int_{\R^N} \Phi(h(y,\tau)) \FL \phi(y,\tau)\dd y\dd \tau\\
&= \frac{b}{a^{2s}} \int_0^{T} \int_{\R^N} \Phi(h(ax,bt)) \FL\psi(x,t) \dd x\dd t,
\end{split}
\end{equation*}
and the choice $b=a^{2s}$ shows identity \eqref{defeq:distSolFL} for $h_a$. We also have,
\begin{equation*}
\begin{split}
\esslim_{t\to0^+} \int_{\R^N} h(a x, bt) \psi(x,t)\dd x
&=\esslim_{\tau\to0^+} \frac{1}{a^N}\int_{\R^N} h(y, \tau) \psi(\frac{y}{a},\frac{\tau}{b})\dd y\\
&=\frac{1}{a^N}\int_{\R^N} h_0(y)\psi(\frac{y}{a},0)\dd y=\int_{\R^N} h_0(ax) \psi(x,0) \dd x.
\end{split}
\end{equation*}
which finishes the proof.
\end{proof}
Once we know the selfsimilarity of the equation, choosing an appropriate initial data will provide us with a selfsimilar solution and a selfsimilar profile.

\begin{lemma}[Existence and uniqueness of a selfsimilar solution]\label{lem:SelfSimilarStructure}
Let $h\in L^\infty(Q_T)$ be the unique very weak solution of \eqref{P1}--\eqref{P1-Init} with initial data $h_0\in L^\infty(\R^N)$. If $h_0(x)=h_0(ax)$ for all $a>0$ and a.e. $x\in \R^N$, then
\[
h(x,t)= h(ax, a^{2s} t) \qquad \textup{for a.e.} \qquad (x,t)\in Q_T \quad \textup{and all} \quad a>0.
\]
In particular, $h(x,t)=h(xt^{-1/(2s)},1)$ for a.e. $(x,t)\in Q_T$.
\end{lemma}

\begin{proof}
We recall that given an initial data $h_0\in L^\infty(\R^N)$, existence and uniqueness of $h$ is given by Theorem \ref{thm:bdddistsol}.

By Lemma \ref{lem:ssofequation} we have that $h(ax,a^{2s}x)$ is a very weak solution of \eqref{P1}--\eqref{P1-Init} with initial data $h_0(ax)$ for all $a>0$. By the scaling property $h_0(ax)=h_0(x)$ we thus get that $h(ax,a^{2s}t)$ is a solution with data $h_0(x)$ for all $a>0$. Uniqueness of very weak solutions implies that $h(x,t)=h(ax,a^{2s}t)$ for all $a>0$. In particular, by choosing $a=t^{-1/(2s)}$ we get the identity.
\end{proof}

We can express the selfsimilar solution in terms of a profile satisfying a stationary equation.

\begin{lemma}[Equation of the profile]\label{lem:profileH2}
Under the assumptions of Lemma \ref{lem:SelfSimilarStructure}, consider the function $H(\xi):=h(\xi,1)$ for a.e. $\xi \in \R^N$. Then we have that $H$ satisfies
\begin{equation}\label{eq:profH}
-\frac{1}{2s} \xi\cdot \nabla H(\xi)+ \FL \Phi(H)(\xi)=0 \qquad \textup{in} \qquad  \mathcal{D}'(\R^N).
\end{equation}

\end{lemma}
\begin{remark}
Note that, by Lemma  \ref{lem:SelfSimilarStructure},  we can express the selfsimilar solution in terms of the profile as $h(x,t)=H(x t^{-1/(2s)})$.
\end{remark}
\begin{proof}[Proof of Lemma \ref{lem:profileH2}]
Let $\xi=xt^{-1/(2s)}$. Formally, we have
\[
\begin{split}
\partial_t h(x,t)=\partial_t\big(H(xt^{-\frac{1}{2s}})\big)=-\frac{1}{2s}t^{-\frac{1}{2s}-1}x\cdot \nabla H(xt^{-\frac{1}{2s}})=-\frac{1}{2s} t^{-1} \xi \cdot \nabla H(\xi),
\end{split}
\]
and $\FL [\Phi(h(\cdot,t)](x)= \FL [\Phi(H(\cdot \, t^{-1/(2s)} ))](x)=t^{-1} \FL \Phi(H)(\xi)$.
Similar computations can be done in distributional sense as in the proof of Lemma \ref{lem:ssofequation}.
\end{proof}

Note that the relation $u(x,t)=(h(x,t)-L)_+$ allows us to also define a selfsimilar profile for $u$. More precisely:
\[
u(x,t)=u(xt^{-\frac{1}{2s}},1)=:U(xt^{-\frac{1}{2s}}).
\]
By combining \eqref{P3} and  \eqref{eq:profH}, we obtain
\[
-\frac{1}{2s} \xi\cdot \nabla H(\xi)+ \FL U(\xi)=0 \qquad \textup{in} \qquad  \mathcal{D}'(\R^N).
\]

The following result shows that whenever $h$ and $u$ (resp. $H$ and $U$) are regular enough in a localized area, the above equations are satisfied in the classical sense.

\begin{lemma}\label{lem:clasSol}
Under the assumptions of Lemma \ref{lem:SelfSimilarStructure}, and, additionally, that $h\in C^1_b(\Lambda)$ and $u=\Phi(h)\in C^2_b(\Lambda)$ for some open set $\Lambda \subset \R^N \times(0,+\infty)$, we have that $\dell_th(x,t)$ and $\FL u(x,t)$ exist and
\[
\dell_th(x,t) + \FL u(x,t)=0 \qquad \textup{for all} \qquad (x,t)\in \Lambda.
\]
In the same way, if $H\in C_\textup{b}^1(\tilde{\Lambda})$ and $U\in C_\textup{b}^2(\tilde{\Lambda})$ for some open set $\tilde{\Lambda}\subset \R^N$,  then
\[
-\frac{1}{2s} \xi\cdot \nabla H(\xi)+ \FL U(\xi)=0  \qquad \textup{for all} \qquad \xi \in \tilde{\Lambda}.
\]
\end{lemma}
\begin{proof}
Let $(x,t)\in \Lambda$. Then there exists $r>0$ such that $B_r((x,t))\subset \subset \Lambda $. Thus
$h\in C^1_b(B_r(x,t))$ and $u=\Phi(h)\in C^2_b(B_r(x,t))$.
Clearly $\dell_th(x,t)$ exists. Moreover,
\[
\begin{split}
|\FL u(x,t)|&\leq\left| \int_{B_r(x)} \frac{u(x,t)-u(y,t)}{|x-y|^{N+2s}} \dd y\right|+\left| \int_{\R^N\setminus B_r(x)} \frac{u(x,t)-u(y,t)}{|x-y|^{N+2s}} \dd y\right|\\
& \leq C_1 \|D^2 u\|_{L^\infty(B_r(x,t))} r^{2-2s}+ C_2\|u\|_{L^\infty(Q_T)} r^{-2s}<+\infty,
\end{split}
\]
so $\FL u(x,t)$ exists too. Consider now a test function $\psi \in C_\textup{c}^\infty(\Lambda)$. Integrating by parts in Definition \ref{def:distSolfrac} gives
\[
\iint_{\Lambda}\big(\dell_th(x,t) \psi(x,t)-\FL u(x,t)  \psi(x,t) \big)\dd x \dd t=0 \quad \textup{for all} \quad \psi \in C_\textup{c}^\infty(\Lambda).
\]
As a consequence we get that $\dell_th(x,t)+\FL u(x,t)=0$ for all $(x,t)\in \Lambda$. The results on $H$ and $U$ follow in a similar way.
\end{proof}

\subsection{Profile, boundedness and limits of  selfsimilar solutions}
As mentioned before, here and in the following sections, we restrict ourselves to the 1-D case. First we start by showing Theorem \ref{thm:SS-all}\eqref{thm-SS-item1}.

\begin{lemma}
Under the assumptions of Theorem \ref{thm:SS-all}, we have that $H$ satisfies \eqref{P2SS} and  $H(x t^{-1/(2s)})=h(x,t)$ for a.e. $(x,t)\in \R\times(0,T)$.
\end{lemma}
\begin{proof}
Note that $h_0(ax)=h_0(x)$ for all $x\in \R$ and all $a>0$. Then the result follows from Lemma \ref{lem:SelfSimilarStructure} and Lemma \ref{lem:profileH2}.
\end{proof}

We prove now the first part of Theorem \ref{thm:SS-all}\eqref{thm-SS-item2}.

\begin{lemma}\label{lem:abovebelow}
Under the assumptions of Theorem \ref{thm:SS-all}, we have that $L-P_2\leq H(\xi)\leq L+P_1$ for a.e. $\xi\in\R$.
\end{lemma}
\begin{proof}
Note that $L+P_1$ and $L-P_2$ are stationary solutions of \eqref{P1}--\eqref{P1-Init}. Clearly $L-P_2\leq h_0 \leq L+P_1$. Thus, by comparison (Theorem \ref{thm:bdddistsol}\eqref{thmfraclinf2-item2}) we have that $L-P_2\leq h\leq L+P_1$ which by definition implies $L-P_2\leq H(\xi)\leq L+P_1$ for a.e. $\xi\in \R$.
\end{proof}

Before proving that the limits of $H$ as $\xi\to\pm \infty$ are taken as stated in the second part of Theorem \ref{thm:SS-all}\eqref{thm-SS-item2}, we need to prove preliminary monotonicity results.

\begin{lemma}\label{lem:monoton}
Under the assumptions of Theorem \ref{thm:SS-all}, we have that $H$ is nonincreasing.
\end{lemma}

\begin{proof}
Consider the initial data $\hat{h}_0(x):=h_0(x+a)$ for some $a>0$ and the corresponding solution $\hat{h}$. Note that $h_0\geq \hat{h}_0$, and thus, by comparison (cf. Theorem \ref{thm:bdddistsol}\eqref{thmfraclinf2-item2}) $h\geq \hat{h}$.  Note also that \eqref{P1} is translationally invariant, and so $h(x+a,t)=\hat{h}(x,t)$ for a.e. $(x,t)\in \R\times(0,T)$. We have then that $h(x,t)\geq h(x+a,t)$ for a.e. $(x,t)\in \R\times(0,T)$, which concludes that $H$ is nonincreasing since $a>0$ was arbitrary and the relation $H(\xi)=h(\xi,1)$ holds.
\end{proof}

We are now ready to prove  the second part of Theorem \ref{thm:SS-all}\eqref{thm-SS-item2}.

\begin{lemma}\label{lem:lim}
Under the assumptions of Theorem \ref{thm:SS-all}, we have that
\[
\displaystyle\lim_{\xi\to-\infty}H(\xi)=L+P_1 \qquad  \textup{and} \qquad \displaystyle\lim_{\xi\to+\infty}H(\xi)=L-P_2.
\]
\end{lemma}
\begin{proof}
By Lemma \ref{lem:abovebelow}, we know that $H$ is bounded from above and from below. By Lemma \ref{lem:monoton}, $H$ is monotone (nonincreasing as $\xi\to+\infty$ and nondecreasing as $\xi\to-\infty$), and this implies the existence of the limits:
\[
\underline{H}:=\lim_{\xi\to-\infty}H(\xi) \quad \textup{and} \quad \overline{H}:=\lim_{\xi\to+\infty}H(\xi).
\]
In particular,
\[
\esslim_{t\to0^+}H(x t^{-\frac{1}{2s}})=\begin{cases}
\underline{H} &\textup{if}\quad    x<0,\\
\overline{H}  &\textup{if}\quad   x>0.
\end{cases}
\]
Given any $\psi \in C_\textup{c}^\infty(\R \times [0,T))$, we have by Definition \ref{def:distSolfrac} that
\begin{equation*}
\esslim_{t\to0^+}\int_{\R}H(x t^{-\frac{1}{2s}})\psi(x,t)\dd x=\esslim_{t\to0^+}\int_{\R}h(x,t)\psi(x,t)\dd x=\int_{\R}h_0(x)\psi(x,0)\dd x.
\end{equation*}
Take now $\psi \in C_\textup{c}^\infty(\R_+\times[0,T))$. The Lebesgue dominated convergence theorem yields
\begin{equation*}
\esslim_{t\to0^+}\int_{\R}H(x t^{-\frac{1}{2s}})\psi(x,t)\dd x=
\int_{\R}\esslim_{t\to0^+} \big(H(x t^{-\frac{1}{2s}})\psi(x,t)\big)\dd x=
\int_{\R_+}\overline{H}\psi(x,0)\dd x,
\end{equation*}
and thus, we have the following identity for all $\phi\in C_\textup{c}^\infty(\R_+)$:
\[
\int_{\R_+}\overline{H}\phi(x)\dd x=\int_{\R}h_0(x)\phi(x)\dd x.
\]
This implies that $\overline{H}=h_0(x)$ for a.e. $x\in \R_+$, i.e. $\overline{H}=L-P_2$. The same argument taking $\psi \in C_\textup{c}^\infty(\R_-\times[0,T))$ shows that $\underline{H}=h_0(x)$ for a.e. $x\in \R_-$, i.e $\underline{H}=L+P_1$.
\end{proof}

\subsection{Strict monotonicity in the ice region and existence of a unique interface point}\label{sec:stricmon}

In this section we will prove that there exists one and only one interface point from water to ice regions (i.e. the first part of Theorem \ref{thm:SS-all}\eqref{thm-SS-item3}), and that in the ice region the solution $h$ is strictly decreasing (Theorem \ref{thm:SS-all}\eqref{thm-SS-item5}).

We start by showing that the temperature function $U$ is a continuous function.

\begin{lemma}\label{lem:Ucont}
Under the assumptions of Theorem \ref{thm:SS-all}, we have that $U\in C_\textup{b}(\R)$ and $0\leq U(\xi) \leq P_1$ for all $\xi \in \R$.
\end{lemma}

\begin{proof}
 Theorem \ref{thmfrac-cont}  ensures that given a solution $h\in L^\infty(\R\times(0,T))$ of \eqref{P1}--\eqref{P1-Init}  with initial data $h_0\in L^\infty(\R)$  we have that
\[
(x,t)\mapsto (h(x,t)-L)_+=\big(H(x t^{-\frac{1}{2s}})-L\big)_+= U(x t^{-\frac{1}{2s}})
\]
is continuous  for $t>0$. In particular, taking $t=1$ we get that $U$ is continuous in $\R$. By Lemma \ref{lem:abovebelow} we also have that
\[
0\leq U(\xi) = (H(\xi)-L)_+\leq (L+P_1-L)_+=P_1. \qedhere
\]
\end{proof}

For now, we know that $H$ is nonincreasing in $\R$ (Lemma \ref{lem:monoton}). However, we will be able to prove that $H$ is strictly decreasing in the ice region (where the temperature is zero).

\begin{lemma}\label{lem:setD}
Under the assumptions of Theorem \ref{thm:SS-all}, we have  that $H$ is strictly decreasing in the nonempty closed set
\[
D:=\{\xi\in \R : H(\xi)\leq L \}= \{\xi\in \R : U(\xi)= 0 \}.
\]
\end{lemma}
\begin{proof}
The set $D$ is nonempty since $\lim_{\xi\to+\infty} H(\xi)=L-P_2<L$ by Lemma \ref{lem:lim}. It is also closed since $U$ is a continuous function by Lemma \ref{lem:Ucont}.

Assume by contradiction that $H$ is not strictly decreasing in $D$. Then there exists two points $\xi_m$ and $\xi_M$ with $\xi_m<\xi_M$ such that $H(\xi_m)=H(\xi_M)=l$ for some $l\leq L$. Since $H$ is nonincreasing, we have
\[
H(\xi)=l \qquad \textup{for all} \qquad \xi \in [\xi_m,\xi_M]
\]
and also that $U(\xi)=0$ for all $\xi \in [\xi_m,+\infty)$. Since both $H$ and $U$ are regular in the set $(\xi_m, \xi_M)$ we can apply Lemma \ref{lem:clasSol} and get that
\[
-\frac{1}{2s} \xi H'(\xi)+ \FL U(\xi)=0 \qquad \textup{for all} \qquad \xi \in (\xi_m, \xi_M).
\]
Note that, additionally, $H'(\xi)=0$ for all $\xi \in (\xi_m, \xi_M)$. Now, we fix $\hat{\xi} \in (\xi_m,\xi_M)$ and use the previous identity to get:
\[
0=\FL U(\hat{\xi})= c_{1,s}\textup{P.V.}\int_{|\eta|>0} \frac{U(\hat{\xi})-U(\eta)}{|\hat{\xi}-\eta|^{1+2s}}\dd \eta=-c_{1,s}\int_{-\infty}^{\xi_m}\frac{U(\eta)}{(\hat{\xi}-\eta)^{1+2s}}\dd \eta.
\]
Since $U$ is nonnegative and continuous the above identity implies that $U(\xi)=0$ for all $\xi \in (-\infty,\xi_m)$ and thus $U\equiv0$ in $\R$. In particular, we have that $H(\xi)\leq L$ for all $\xi \in \R$, which is a contradiction with the fact that $\lim_{\xi\to-\infty} H(\xi)=L+P_1>L$.
\end{proof}

We are now ready to prove the existence of a unique interface point, and that this point lies in the half line $[0,+\infty)$. Moreover we will also get the strict monotonicity stated in Theorem \ref{thm:SS-all}\eqref{thm-SS-item5} concluding the proof of that part of the theorem.

We will still be missing the fact that $\xi_0>0$, which will ensure that the free boundary of the selfsimilar solution advance as the time passes. This will require a more delicate study.

\begin{figure}[h!]\center
\includegraphics[width=\textwidth]{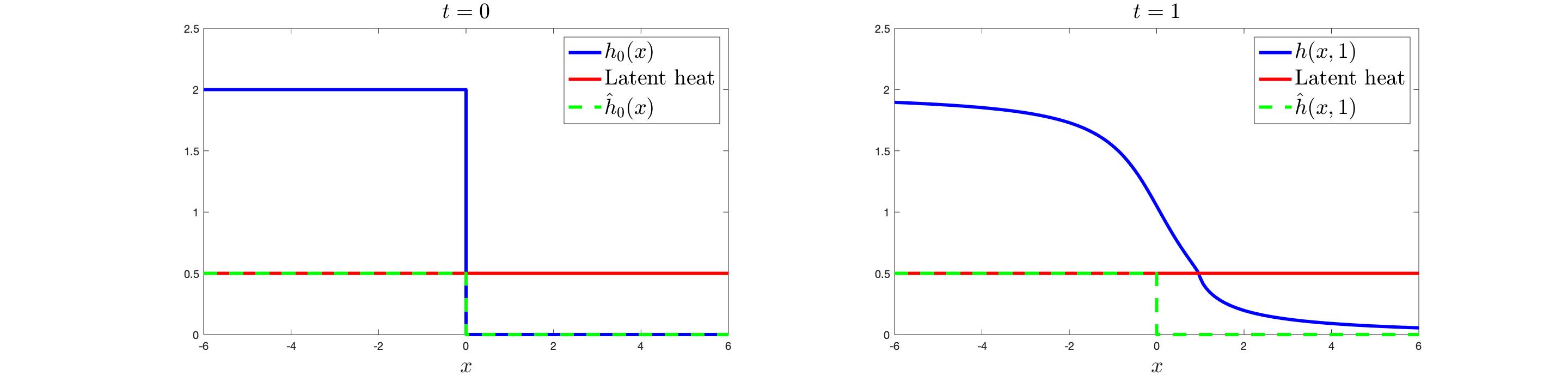}
\caption{Graphical proof of the nonnegativity of the interface point.}
\label{fig:Comparisoninterface}
\end{figure}

\begin{lemma}\label{eq:uniqueinterface}
Under the assumptions of Theorem \ref{thm:SS-all}, we have that there exists a unique $\xi_0\in \R$ such that $H(\xi_0)=L$. Moreover, $0\leq \xi_0<+\infty$ and $H$ is strictly decreasing in $[\xi_0,+\infty)$.
\end{lemma}
\begin{proof}
First we will show that there exists a unique $\xi_0\in \R$ such that $H(\xi_0)=L$. By Lemma \ref{lem:lim} $\lim_{\xi\to+\infty}H(\xi)=L-P_2<L$ and thus, there exists at least one point $\xi_1<+\infty$ such that $H(\xi_1)<L$, which in particular implies that $U(\xi_1):=(H(\xi_1)-L)_+=0$. Note that $U(\xi)=(H(\xi)-L)_+$ is nonincreasing, continuous and $\lim_{\xi\to-\infty}U(\xi)=P_1>0$. Thus, there exists $\xi_0\in \R$ such that
\[
\xi_0=\inf\{\xi\in \R \, :\, U(\xi)=0\}<+\infty.
\]
Then, for all $\xi<\xi_0$ we have that $U(\xi)>0$ and so $H(\xi)>L$. This implies that $H$ is continuous in $(-\infty,\xi_0]$ since $H(\xi)=U(\xi)+L$ there. We conclude then that $H(\xi_0)=L$.

With the above argument, we have also found that the set $D$ from Lemma \ref{lem:setD} is given by
\[
D=[\xi_0,+\infty).
\]
By strict monotonicity of $H$ in the set $D$ we have also that $H(\xi)=L$ only when $\xi=\xi_0$. This shows the existence of a unique interface point $\xi_0$.

Now we will prove that $\xi_0\geq0$. Consider
\begin{equation*}
\hat{h}_0(x):=\begin{cases}
L  \qquad\qquad&\textup{if}\qquad   x\leq0\\
0  \qquad\qquad&\textup{if}\qquad   x>0,
\end{cases}
\quad \textup{and} \quad
\hat{h}(x,t):=\begin{cases}
L  \qquad\qquad&\textup{if}\qquad   x\leq0\\
0  \qquad\qquad&\textup{if}\qquad   x>0,
\end{cases}
\end{equation*}
Note that $\hat{u}(x,t):=(\hat{h}(x,t)-L)_+=0$ for all $(x,t)\in \R\times(0,T)$ and thus $\hat{h}$ is a solution of \eqref{P1}--\eqref{P1-Init} with initial data $\hat{h}_0$. Since $\hat{h}_0\leq h_0$ in $\R$, we have by comparison (Theorem \ref{thm:bdddistsol}\eqref{thmfraclinf2-item2}) that $\hat{h}\leq h$ a.e. in $Q_T$ (see Figure \ref{fig:Comparisoninterface}). Then
\[
H(\xi)=h(\xi,1)\geq \hat{h}(\xi,1)=\hat{h}_0(\xi) \quad\textup{for a.e.}\quad \xi\in\R.
\]
Since $\hat{h}_0(0)=L$, the continuity of $(H(\xi)-L)_+$ gives $H(0)\geq L$.
On one hand, if $H(0)>L$, we have by monotonicity and continuity of $H$
that $\xi_0>0$. On the other hand, if $H(0)=L$, we have by the uniqueness of the interface point proved before that $\xi_0=0$. Thus, $\xi_0\geq0$.
\end{proof}

\subsection{Regularity in the ice and water regions}
We already know that $U$ is a continuous function, which together with the results in Section \ref{sec:stricmon} implies that $H$ is continuous in the water region, i.e. $H\in C_\textup{b}((-\infty,\xi_0])$. The regularity result in the water region can be improved as follows.

\begin{lemma}\label{lem:regwater}
Under the assumptions of Theorem \ref{thm:SS-all},  we have that $H\in C^{1,\alpha}((-\infty,\xi_0))$ for some $\alpha\in(0,1)$.
\end{lemma}

\begin{proof}
Let us examine the regularity of $U$ in the water region $(-\infty,\xi_0)$, i.e., in terms of space-time variables, the regularity of $u$ in the region $P:=\{(x,t): \ t>0, \ -\infty<x<\xi_0t^{1/(2s)}\}$. We localize around a point $(x_0,t_0)\in P$ and consider a small parabolic cylinder $Q_1(x_0,t_0)=(x_0-r,x_0+r)\times (t_0-\tau, t_0)$ where  $u$ is a continuous and bounded solution of $\partial_tu+(-\Delta)^su=0$. Moreover, $u$ is also bounded in $\R\times (t_0-\tau, t_0)$. By Theorem 4.1 of \cite{Silv2012} we conclude that $u\in C_{x,t}^{\alpha, \alpha/(2s)}(Q_2)$ in a certain subcylinder $Q_2\subset Q_1$, with a fixed $\alpha>0$ and constants that do not depend on $(x_0,t_0)\in P$ as long as we are away from the free boundary $x=\xi_0 t^{1/(2s)}$. Thus, $U$ is in $ C^\alpha((-\infty,\xi_0))$ uniformly away from $\xi_0$. See also \cite{ChangLD2012}. The improvement of regularity is obtained in Theorem 6.2 in \cite{ChDa14} (see also Theorem 2.4 of \cite{ChangLD2014}) where
it is proved that (for more general parabolic integro-differential operators) $u$ is $C^{1,\alpha}$ in space for some universal $\alpha \in (0, 1)$ inside the same water region. This proves that
$U\in C^{1,\alpha}((-\infty,\xi_0))$ and, since $H=U+L$ in $(-\infty,\xi_0)$, we are done.  We also conclude that equation \eqref{P2SS} is satisfied in the classical sense in $(-\infty,\xi_0)$ since $H'\in C^{\alpha}((-\infty,\xi_0))$ and thus, $\FL U  \in C^{\alpha}((-\infty,\xi_0))$ from \eqref{P2SS}.
\end{proof}

Next, we  study regularity in the ice region $I=(\xi_0,+\infty)$. It turns out to be better.

\begin{lemma}\label{lem:estimHice}
Under the assumptions of Theorem \ref{thm:SS-all}, we have that $H\in C^{\infty}((\xi_0,+\infty))$ and $H\in C^{\infty}_\textup{b}([\hat{\xi},+\infty))$ for all $\hat{\xi}>\xi_0$. Moreover:
\begin{enumerate}[\rm (a)]
\item\label{LemHice-item1}$ \xi H'(\xi)= 2s \FL U (\xi)$ for all $\xi>\xi_0$.
\item\label{LemHice-item2}  $H'(\xi)\asymp - 1/|\xi|^{1+2s}$ for all $\xi\gg \xi_0$.
\item\label{LemHice-item3} Given any $\xi_1, \xi_2$ such that $\xi_0<\xi_1<\xi_2$, we have that $H$ satisfies
 \begin{equation}\label{est:xipos}
H(\xi_2)-H(\xi_1)= 2s \int_{\xi_1}^{\xi_2} \frac{\FL U(\xi)}{\xi}\dd \xi.
\end{equation}
\end{enumerate}
\end{lemma}

\begin{proof}
By Lemma \ref{lem:Uice} we have that $\FL U \in C^{\infty}((\xi_0,+\infty))$. The distributional identity \eqref{P2SS} together with the mentioned regularity imply that $\xi H'(\xi)$ is also a function and
\[
\xi H'(\xi)=2s \FL U (\xi)\qquad\textup{for all}\qquad \xi>\xi_0.
\]
Moreover, since both $1/\xi$ and $\FL U(\xi)$ are infinitely smooth in $(\xi_0,\infty)$ and also with bounded derivatives in $[\hat{\xi},+\infty)$, we get that $H\in C^{\infty}((\xi_0,+\infty))$ and $H\in C^{\infty}_\textup{b}([\hat{\xi},+\infty))$. To prove part \eqref{LemHice-item2}, we use Lemma \ref{lem:Uice} to get that for $\xi\gg\xi_0$ we have
\[
H'(\xi) =2s\frac{\FL U(\xi)}{\xi}\asymp - \frac{1}{|\xi-\xi_0|^{2s}}\frac{1}{\xi} \asymp- \frac{1}{|\xi|^{1+2s}}
\]
This proves the second estimate in Theorem \ref{thm:SS-all}\eqref{thm-SS-item7}. Identity \eqref{est:xipos} follows from part \eqref{LemHice-item1}.
\end{proof}
The regularity at the interface $\xi_0$ is still missing here. It will be discussed in Section \ref{sec:continterface}.

\subsection{Strict positivity of the interface point}

It is crucial to prove the strict positivity of the interface region, i.e., that $\xi_0>0$. A consequence of this fact is that the free boundary of the selfsimilar solution is moving forward in time.

\begin{lemma}\label{lem:posinterp}
Under the assumptions of Theorem \ref{thm:SS-all}, we have that $\xi_0>0$.
\end{lemma}
\begin{proof}
We argue by contraction: Under the assumption $\xi_0=0$, the strategy of the proof is to show that the behaviour of $U$ close to the interface has a certain lower bound, and this lower bound will imply that the solution itself is not bounded in $[\xi_0,+\infty)=[0,+\infty)$. Thus, reaching a contradiction.

\smallskip
\noindent\textbf{1)} \emph{Lower bound of $U$}. We want to prove that if $\xi_0=0$, then $U(\xi) \gtrsim |\xi|^s$ for all $\xi<0$ close enough to $\xi_0=0$.

Fix $\hat{\xi}$, consider $I:=[\hat{\xi},0]$, and let $U^I$ solve (cf. \eqref{P2SS})
\begin{equation}\label{P2SSI}
\begin{cases}
\FL U^I (\xi)=\frac{1}{2s}\xi H'(\xi) \qquad\qquad&\text{in}\qquad\xi \in I,\\
U^I (\xi)=0  \qquad\qquad&\text{in}\qquad \xi\in I^{c}.
\end{cases}
\end{equation}
If we are able to get $U^I(\xi)\gtrsim |\xi|^{s}$ for all $\xi\in I$, then by linearity and the fact that $U\geq0$ in $\R$, $U-U^I$ solves
\begin{equation*}
\begin{cases}
\FL (U-U^I)(\xi)=0 \qquad\qquad&\text{in}\qquad\xi \in I,\\
(U-U^I)(\xi)\geq0  \qquad\qquad&\text{in}\qquad \xi\in I^{c}.
\end{cases}
\end{equation*}
By the strong maximum principle, it follows that $U(\xi)\geq U^I(\xi)$ for all $\xi\in I$, and we get
\begin{equation}\label{eq:upbound}
U(\xi)\geq U^I(\xi)\gtrsim |\xi|^{s}.
\end{equation}

Let us then continue by proving the lower bound for $U^I$.  Recall that $H'\in C((-\infty,0))$ (Lemma \ref{lem:regwater}), $H'\leq0$ (Lemma \ref{lem:monoton}), and $\|H'\|_{L^1((-\infty,0))}=P_1$. We can moreover assume that $\hat{\xi}$ is such that $H'\not\equiv0$. Since the right-hand side in \eqref{P2SSI} is $L^1$, we know that the unique solution $U^I$ is given by \cite{ChenVeron, GC-Va17}
$$
U^I(\xi)=\frac{1}{2s}\int_{\hat{\xi}}^{0} \mathbb{G}(\xi,y) y H'(y)\dd y\not\equiv0.
$$
The right-hand side is also nonnegative and then we conclude by the Hopf lemma \cite[Lemma 7.3]{R-Ot16} (see also \cite{R-OtSe14}) that $U^I(\xi)\gtrsim |\xi|^{s}$.

\smallskip
\noindent\textbf{2)} \emph{Contradiction on $L<+\infty$}.
Recall that for $\xi>\xi_0=0$ we have that $U(\xi)=0$. Thus, using \eqref{eq:upbound}, we have for all $\xi>0$ close enough to zero, that
\[
\begin{split}
-\FL U(\xi)&=c_{1,s} \int_{-\infty}^0 \frac{U(\eta)}{|\eta-\xi|^{1+2s}}\dd \eta \gtrsim  \int_{-2\xi}^{-\xi} \frac{|\eta|^{s}}{|\eta-\xi|^{1+2s}}\dd \eta\geq  \int_{-2\xi}^{-\xi} \frac{\dd \eta}{|\eta-\xi|^{1+s}}\sim \frac{1}{|\xi|^{s}}.
\end{split}
\]
We use now estimate \eqref{est:xipos}, the fact that $H\geq L-P_2$, and the above upper bound to get
\begin{equation*}
\begin{split}
H(\xi_1)&= H(\xi_2) - 2s \int_{\xi_1}^{\xi_2}  \frac{\FL U(\eta)}{\eta}\dd \eta
 \gtrsim L-P_2+ \int_{\xi_1}^{\xi_2} \frac{\dd \eta}{\eta^{1+s}}.
\end{split}
\end{equation*}
Recall that $H\leq L$ in $[\xi_0=0,+\infty)$. Taking limits as $\xi_1\to0^+$ in the above estimate we get
\[
L\geq \lim_{\xi_1\to0^+} H(\xi_1) \gtrsim L-P_2+ \int_0^{\xi_2} \frac{\dd \eta}{\eta^{1+s}}=+\infty.
\]
which is a contradiction and proves that $\xi_0>0$.
 \end{proof}

 With this result we have finished the proof of Theorem \ref{thm:SS-all}\eqref{thm-SS-item3}.

 \subsection{Continuity of the enthalpy and behaviour at the interface}\label{sec:continterface}
 To conclude that $H$ is continuous in $\R$, we only have to check that there exists no discontinuity at the interface, that is,
 \[
 \lim_{\xi\to\xi_0^-} H(\xi)=L= \lim_{\xi\to\xi_0^+} H(\xi).
 \]
Recall that the left-hand side of this identity was a consequence of the continuity of $U$ (cf. Lemma \ref{eq:uniqueinterface}). To prove the other one, we first need to show Theorem \ref{thm:SS-all}\eqref{thm-SS-item6}.

\begin{lemma}\label{lem:upperbound}
Under the assumptions of Theorem \ref{thm:SS-all}, we have that
\begin{equation}\label{eq:upperest2}
U(\xi)=H(\xi)-L=O\left((\xi_0-\xi)^{s}\right) \quad \textup{for all} \quad \xi\leq \xi_0.
\end{equation}
\end{lemma}

\begin{proof}
The strategy is to find a suitable upper barrier for $U$ that  itself  satisfies estimate \eqref{eq:upperest2}. To do so, we will first note that $u$ satisfies a fractional heat equation with a time-dependent domain $(-\infty, \xi_0 t^{1/(2s)})$ (which is expanding in time). A supersolution for this problem will be given by the solution of the fractional heat equation in the domain $(-\infty, \xi_0)$.

Let $v$ be the solution of
\begin{equation*}
\begin{cases}
\dell_tv+\FL v=0 \qquad\qquad&\textup{in}\qquad   (-\infty,\xi_0) \times (0,1],\\
v=0\qquad\qquad&\textup{in}\qquad   [\xi_0,+\infty)\times[0,1],\\
v(\cdot,0)=u_0 \qquad\qquad  &\textup{in}\qquad   (-\infty,\xi_0).
\end{cases}
\end{equation*}
On one hand, we know from the results of \cite{ChKiSo10,BoGrRy10} (see also discussion in \cite{FeRo16}) that
\[
0\leq v(x,t)\lesssim |x-\xi_0|^{s}  \quad \textup{for all} \quad x\leq \xi_0
\]
On the other hand, we know that $U(x)=u(x,1)$ where $u$ satisfies
\begin{equation*}
\begin{cases}
\dell_tu+\FL u=0 \qquad\qquad&\textup{in}\qquad    (-\infty, \xi_0 t^{\frac{1}{2s}})\times(0,1],\\
u=0\qquad\qquad&\textup{in}\qquad   [\xi_0 t^{\frac{1}{2s}},+\infty)\times[0,1],\\
u(\cdot,0)=u_0 \qquad\qquad&\textup{in}\qquad   (-\infty,\xi_0).
\end{cases}
\end{equation*}
Note also that in the region $(x,t)\in[\xi_0t^{1/(2s)}, \xi_0]\times(0,1]$, we have that $u=0$ (and $u\geq0$ in $\R$) and thus $\dell_tu=0$ and $\FL u\leq0$ there, i.e,
\[
\partial_tu+\FL u\leq0 \quad\quad \textup{in} \qquad [\xi_0t^{\frac{1}{2s}}, \xi_0]\times(0,1].
\]
To finish, we consider $w=v-u$. It satisfies:
\begin{equation*}
\begin{cases}
\partial_tw+\FL w\geq0 \qquad\qquad&\textup{in}\qquad   (-\infty,\xi_0) \times (0,1],\\
w=0\qquad\qquad&\textup{in}\qquad   [\xi_0,+\infty)\times[0,1],\\
w(\cdot,0)=0 \qquad\qquad   &\textup{in}\qquad   (-\infty,\xi_0).
\end{cases}
\end{equation*}
Thus, $w\geq0$ (see \cite{FeRo16}) and so $
u(x,t)\leq v(x,t)\lesssim |x-\xi_0|^{s}$. In particular, $U(\xi)=u(\xi,1)\lesssim |\xi-\xi_0|^{s}$.
\end{proof}

We are now ready to finish the proof of Theorem \ref{thm:SS-all}\eqref{thm-SS-item4} by proving the continuity of $H$ at the interface.

\begin{lemma}
Under the assumptions of Theorem \ref{thm:SS-all}, we have that
\[
 \lim_{\xi\to\xi_0^-} H(\xi)=L= \lim_{\xi\to\xi_0^+} H(\xi).
 \]
\end{lemma}
\begin{proof}
By the continuity of $U$  and Lemma \ref{eq:uniqueinterface}, we know that
\[
H(\xi_0^-):=\lim_{\xi\to \xi_0^-} H(\xi)=L.
\]
We need to show that $H(\xi_0^+):=\lim_{\xi\to \xi_0^+} H(\xi)=L$. For now,  assume $H(\xi_0^+)=L-A$ for some $A\in[0,L-P_2]$. We will prove that $A=0$.

Let $\eta_\veps$ be the cut-off functions defined in Appendix \ref{sec.aux} with $R=\veps$, $x=\xi-\xi_0$.
Take it as a test function in the weak formulation of \eqref{P2SS} to get
\begin{equation}\label{eq:helpcont}
\int_{\xi_0-\veps}^{\xi_0+\veps} H(\xi)(\xi \eta_\veps(\xi-\xi_0))' \dd \xi=-\int_{-\infty}^{\xi_0} U(\xi) \FL [\eta_{\veps}(\cdot-\xi_0)](\xi) \dd \xi
\end{equation}
Note that $(\xi\eta_\veps)'=\eta_\veps+\xi\eta_\veps'$. For the left-hand side of \eqref{eq:helpcont}, $H\leq L+P_1$ and $|\eta_{\veps}|\leq 1$ yield
\[
\Bigg|\int_{\xi_0-\veps}^{\xi_0+\veps} H(\xi) \eta_\veps(\xi-\xi_0)\dd \xi\Bigg|\leq(L+P_1)\int_{\xi_0-\veps}^{\xi_0+\veps}\dd \xi=2(L+P_1)\veps \stackrel{\veps\to0^+}{\longrightarrow} 0,
\]
and since $\eta_\veps'$ converges in measure to $-\delta_{\xi_0^+}$ in $(\xi_0,\xi_0+\veps)$ and to $\delta_{\xi_0^-}$ in $(\xi_0-\veps,\xi_0)$ as $\veps\to0^+$ and $H(\xi)\xi$ is $C_\textup{b}((\xi_0-\veps,\xi_0)\cup(\xi_0,\xi_0+\veps))$,
\[
\begin{split}
\int_{\xi_0-\veps}^{\xi_0+\veps} H(\xi)\xi \eta_\veps'(\xi-\xi_0) \dd \xi
&=\frac{1}{\veps}\int_{\xi_0-\veps}^{\xi_0} H(\xi)\xi \eta'\Big(\frac{\xi-\xi_0}{\veps}\Big)\dd \xi+\frac{1}{\veps}\int_{\xi_0}^{\xi_0+\veps} H(\xi)\xi \eta'\Big(\frac{\xi-\xi_0}{\veps}\Big)\dd \xi\\
&\stackrel{\veps\to0^+}{\longrightarrow}(H(\xi_0^-)-H(\xi_0^+))\xi_0=(L-(L-A))\xi_0=A\xi_0.
\end{split}
\]

By Lemma \ref{lem:FLHomogeneity} and the change of variables $\xi=\veps z+\xi_0$, the right-hand side of \eqref{eq:helpcont} yields
\[
\begin{split}
&\int_{-\infty}^{\xi_0} U(\xi) \FL [\eta_{\veps}(\cdot-\xi_0)](\xi) \dd x
=\int_{-\infty}^{\xi_0} U(\xi) \FL \Big[\eta\Big(\frac{\cdot-\xi_0}{\veps}\Big)\Big](\xi) \dd \xi\\
&=\frac{1}{\veps^{2s}}\int_{-\infty}^{\xi_0} U(\xi) \FL \eta\Big(\frac{\xi-\xi_0}{\veps}\Big)\dd \xi
= \veps^{1-2s}\int_{-\infty}^{0} U(\veps z+\xi_0) \FL \eta (z)\dd z.
\end{split}
\]
Note that Lemma \ref{lem:upperbound} gives $U(\veps z+\xi_0)\lesssim |(\veps z + \xi_0) - \xi_0|^{s}=\veps^{s}|z|^{s}$ for all $z<0$. From this estimate and Lemma \ref{lem:CutOff}\eqref{item1:lem:CutOff}, we thus get
\[
\begin{split}
\Bigg|\int_{-\infty}^{\xi_0} U(\xi) \FL \eta_{\veps}(\xi+\xi_0) \dd \xi\Bigg|&\lesssim \veps^{1-s} \int_{-\infty}^{0} |z|^{s} |\FL \eta (z)| \dd z   \lesssim \veps^{1-s} \int_{-\infty}^{0} \frac{|z|^{s} }{(1+|z|)^{1+2s}}  \dd z\\
& \lesssim   \veps^{1-s} \stackrel{\veps\to0^+}{\longrightarrow} 0.
\end{split}
\]
In this way, we have proved that $A\xi_0=0$, and since $\xi_0>0$ (cf. Lemma \ref{lem:posinterp}) it implies that $A=0$ and thus $H(\xi_0^+)=L=H(\xi_0^-)$.
\end{proof}
\subsection{Decay of the solution and mass transfer}\label{sec:decayandmasstransfer}

Since we are working with merely bounded solutions, the concept of conservation of mass does not make sense in general. However, in the case where the tail of the solution is an integrable function, we could have the concept of conservation of transferred mass as stated in Theorem \ref{thm:SS-all}\eqref{thm-SS-item8}.

First we prove an estimate on the tail of the solution for $\xi$ large. This will finish the proof Theorem \ref{thm:SS-all}\eqref{thm-SS-item7}.

\begin{lemma}\label{lem:tailbeha}
Under the assumptions of Theorem \ref{thm:SS-all}, we have that
\begin{equation}\label{eq:tailcontrol}
H(\xi)-(L-P_2)\asymp  1/|\xi|^{2s} \qquad \textup{for all}\qquad  \xi\gg\xi_0.
\end{equation}
\end{lemma}

\begin{proof}
Letting $\xi_2\to+\infty$ in \eqref{est:xipos}, and recalling that $H(\xi)\to L-P_2$ as $\xi\to+\infty$, we have for all $\xi>\xi_0$ the following relation
\begin{equation}\label{eq:esttailinf}
H(\xi)=(L-P_2)-2s\int_{\xi}^{\infty} \frac{\FL U(\eta)}{\eta}\dd \eta.
\end{equation}
Recall also that $U(\xi)=0$ for all $\xi\geq \xi_0$, and thus, for all $\xi>\xi_0$,
\[
\begin{split}
H(\xi)-(L-P_2)
&\sim\int_{\xi}^\infty \frac{1}{\eta} \left(\int_{-\infty}^{\xi_0} \frac{U(y)}{|\eta-y|^{1+2s}}\dd y\right)\dd \eta.
\end{split}
\]
To obtain the upper bound in \eqref{eq:tailcontrol}, recall that $U(\xi)\leq P_1$ for all $\xi\in \R$, and thus, for all $\xi\gg\xi_0$, we have that
\[
\begin{split}
H(\xi)-(L-P_2)
&\lesssim P_1 \int_{\xi}^\infty\frac{1}{\eta}  \int_{-\infty}^{\xi_0} \frac{1}{|\eta-y|^{1+2s}}\dd y\dd \eta\\
&\lesssim \int_{\xi}^\infty\frac{1}{\eta} \frac{1}{|\eta-\xi_0|^{2s}}\dd \eta\lesssim \int_{\xi}^\infty\frac{1}{\eta}\frac{1}{\eta^{2s}}\dd\eta
\sim \frac{1}{\xi^{2s}}.
\end{split}
\]
To prove the lower bound, we recall that $\xi_0>0$ is the smallest point where $U=0$ (cf. Lemma \ref{eq:uniqueinterface}), and thus $U(\xi)\geq U(0)>0$ for all $\xi\leq0$. Then \eqref{eq:esttailinf} gives
\[
\begin{split}
H(\xi)-(L-P_2)
&\gtrsim \int_{\xi}^\infty \frac{1}{\eta} \left(\int_{-\infty}^{0} \frac{U(y)}{|\eta-y|^{1+2s}}\dd y\right)\dd \eta\\
&\gtrsim  \int_{\xi}^\infty \frac{1}{\eta} \left(\int_{-\infty}^{0} \frac{1}{|\eta-y|^{1+2s}}\dd y\right)\dd \eta\sim \int_{\xi}^\infty \frac{1}{\eta^{1+2s}}\dd \eta
\sim \frac{1}{\xi^{2s}}.
\end{split}
\]
This finishes the proof.
\end{proof}

\begin{lemma}\label{lem:ConservationOfMass}
Under the assumptions of Theorem \ref{thm:SS-all}, we have that for $s>1/2$
\begin{equation}\label{eq:ConservationOfMass}
\int_{-\infty}^0\big( (L+P_1)- H(\xi)\big)\dd \xi= \int_0^\infty \big(H(\xi)-(L-P_2)\big) \dd \xi<+\infty.
\end{equation}
%For $s\leq1/2$ both integrals above are infinite.
\end{lemma}

\begin{remark}
The case $s\leq 1/2$ is different since Lemma \ref{lem:tailbeha} gives
$$
H(\xi)-(L-P_2)\asymp  1/|\xi|^{2s} \qquad \text{for all}\qquad  \xi\gg\xi_0,
$$
so that the right-hand side of the equality in \eqref{eq:ConservationOfMass} is infinite. Hence, we cannot speak about global mass transfer.
\end{remark}

\begin{proof}[Proof of Lemma \ref{lem:ConservationOfMass}]
Let $\eta_\veps$ be the cut-off functions defined in Appendix \ref{sec.aux}  with $R=\frac{1}{\veps}$. Now, choose $\psi(x,t)=\mathbf{1}_{[0,1]}(t)\eta_\veps(x)$ as test function in Definition \ref{def:distSolfrac}  to obtain
$$
\int_\R \big(h(x,1)-h_{0}(x)\big)\eta_\veps(x)\dd x=-\int_0^1\int_\R u(x,t)\FL\eta_\veps(x)\dd x\dd t.
$$
We can consider this test function, by approximation in $C_\textup{c}^\infty$, using Lemma \ref{lem:CutOff}.
To continue, we change the variables $y=\veps x$ and use that $u\leq P_1$ and Lemmas \ref{lem:FLHomogeneity} and \ref{lem:CutOff} to get
\begin{equation*}%\label{eq:CutOffToZero}
\Bigg|-\int_0^1\int_\R u(x,t)\FL\eta_\veps(x)\dd x\dd t\Bigg|\leq \veps^{2s-1}P_1\|(-\Delta)^s\eta\|_{L^1(\R)}.
\end{equation*}
If $s>1/2$, the above quantity goes to zero when $\veps\to0^+$. Hence,
$$
\lim_{\veps\to0^+}\int_\R \big(h(x,1)-h_0(x)\big)\eta_\veps(x)\dd x=0,
$$
or equivalently,
$$
\lim_{\veps\to0^+}\Bigg(\int_0^\infty \big(H(\xi)-(L-P_2)\big)\eta_\veps(\xi)\dd \xi+\int_{-\infty}^0\big(H(\xi)-(L+P_1)\big)\eta_\veps(\xi)\dd \xi\Bigg)=0.
$$
Note that by the properties of $\eta_\veps$ and Lemma \ref{lem:tailbeha}, the Lebesgue dominated convergence theorem yields
$$
\lim_{\veps\to0^+}\int_0^\infty \big(H(\xi)-(L-P_2)\big)\eta_\veps(\xi)\dd \xi=\int_0^\infty \big(H(\xi)-(L-P_2)\big)\dd \xi<+\infty.
$$
As a consequence,
$$
\lim_{\veps\to0^+}\int_{-\infty}^0\big((L+P_1)-H(\xi)\big)\eta_\veps(\xi)\dd \xi= \int_0^\infty (H(\xi)-(L-P_2))\dd \xi<+\infty.
$$
The sequence $\{\big((L+P_1)-H(\xi)\big)\eta_\veps(\xi)\}_{\veps>0}$ is nonnegative by Lemma \ref{lem:abovebelow} and converges monotonically to $(L+P_1)-H(\xi)$ as $\veps\to0^+$. An application of the monotone convergence theorem then concludes the proof.
\end{proof}

%%%%%%%%%%%%%%%%%%%%%%%%%%%%%%%%%%%%%%%%%%%%%%%%%%%%%%%%%%%%%%%%%%

\section{Propagation properties}
\label{sec.fpp}

Since we have established suitable properties of the selfsimilar solutions, we can prove that the temperature $u(x,t):=\Phi(h(x,t))=(h(x,t)-L)_+$ has finite speed of propagation with precise estimates on the maximal growth of the support in time. We  only need to assume very mild properties on the initial condition $h_0$.

\begin{theorem}[Finite speed of propagation for $u$]\label{coro:NFiniteSpeed2}
Let $h\in L^\infty(Q_T)$ be the very weak solution of \eqref{P1}--\eqref{P1-Init} with $ h_0\in L^\infty(\R^N)$ as initial data and $u:=\Phi(h)$. If $\supp\{\Phi(h_0(x)+\veps)\}\subset B_R(x_0)$ for some $\veps>0$, $R>0$, and $x_0\in \R^N$, then:
\begin{enumerate}[{\rm (a)}]
\item\label{coro:NFiniteSpeed2-item-a}  \textup{(Growth of the support)}  $\supp\{{u(\cdot,t)\}}\subset B_{R+\xi_0 t^{\frac{1}{2s}}}(x_0)$ for all $t\in(0,T)$ and some $\xi_0>0$ depending on the quantity $\veps^{-1}\|\Phi(h_0)\|_{L^\infty(\R^N)}$ and $s$.
\item \textup{(Maximal support)}  $\supp\{{u(\cdot,t)\}}\subset B_{\tilde{R}}(x_0)$ for all  $t\in(0,+\infty)$ with
$$
\tilde{R}= \left(\veps^{-1}\|\Phi(h_0)\|_{L^\infty(\R^N)}+1\right)^{\frac{1}{N}}R.
$$
\end{enumerate}
\end{theorem}

The proof of Theorem \ref{coro:NFiniteSpeed2}\eqref{coro:NFiniteSpeed2-item-a} is based on the one dimensional finite speed of propagation stated in Lemma \ref{thm:NFiniteSpeed2}.

\begin{remark} \rm
Note that the assumption $\veps>0$ in Theorem \ref{coro:NFiniteSpeed2}
is crucial. Otherwise, the result might not be true, as in the following example: Take
\begin{equation*}
h_0(x):=\begin{cases}
L+1  \qquad\qquad&\textup{if}\quad   x\in B_1(0),\\
L  \qquad\qquad&\textup{if}\quad   x\in B_1(0)^c.
\end{cases}
\end{equation*}
Clearly $\supp\{\Phi(h_0)\}=B_1(0)$ but $\supp\{\Phi(h_0+\veps)\}=\R^N$ for all $\veps>0$. Here the property of finite speed of propagation fails since by comparison $h(x,t)\geq L$ in $Q_T$ and thus $u(x,t)=(h(x,t)-L)_+=h(x,t)-L$ and satisfies the fractional heat equation
\[
\partial_tu+\FL u=0 \qquad \textup{in} \qquad Q_T
\]
with $u(x,0)=\indik_{B_1(0)}(x)$.  The solution  $u$ has infinite speed of propagation, i.e. $u(x,t)>0$ for all $t\in(0,T)$ (see e.g. \cite{BoSiVa17}).  See also Figures \ref{fig:instantemergin} and \ref{fig:noninstantemergin} for examples related to this phenomenon.
\end{remark}

\subsection{Proof of the finite propagation theorem}
\label{sec.ppp}
The first part of the proof of Theorem \ref{coro:NFiniteSpeed2} is based on the following 1-D result on the speed of propagation. Recall that $\mathcal{H}$ is a hyperplane in $\R^N$ such that the open sets $\mathcal{H}^+, \mathcal{H}^-$ satisfy $\mathcal{H}^+\cup \mathcal{H}^-=\R^N\setminus \mathcal{H}$ and $\mathcal{H}^+\cap \mathcal{H}^-=\emptyset$ (cf. Remark \ref{rem:hyperplanes}).  We denote $\vec{v}_{\mathcal{H}^{\perp}}$ as the vector orthogonal to  $\mathcal{H}$ with length 1 such that given any $x\in \mathcal{H}$, then $x+\vec{v}_{\mathcal{H}^{\perp}}\in \mathcal{H}^+$.

\begin{lemma}[One-directional finite speed of propagation for $u$]\label{thm:NFiniteSpeed2}
Let $h\in L^\infty(Q_T)$ be the very weak solution of \eqref{P1}--\eqref{P1-Init} with $ h_0\in L^\infty(\R^N)$ as initial data. If $\supp\{\Phi(h_0(x)+\veps) \}\subset \mathcal{H}^-$
for some hyperplane $\mathcal{H}\subset \R^N$ and some $\veps>0$,
then $\supp\{{u(\cdot,t)\}}\subset \mathcal{H}^-+\xi_0 t^{1/(2s)} \vec{v}_{\mathcal{H}^{\perp}}$ for some $\xi_0>0$ and all $t\in(0,T)$.

\end{lemma}

\begin{proof}
Without loss of generality, we can assume that $\mathcal{H}=\{0\}\times \R^{N-1}$ (by translation and rotation invariance of the equation). We take $\mathcal{H}^-=(-\infty,0)\times  \R^{N-1}$ and $\mathcal{H}^+=(0,+\infty)\times  \R^{N-1}$. Consider $\tilde{h}_0$ and $\tilde{h}$ given by Corollary \ref{cor:NSS-all} with $P_1=\esssup_{\R^N}\{\Phi(h_0)\}$ and $P_2=\veps$.  Clearly $h_0 \leq \tilde{h}_0$ and then by comparison $h(x,t) \leq \tilde{h}(x,t)$. Consequently $0\leq u(x,t)\leq \tilde{u}(x,t)=\Phi(\tilde{h}(x,t))$. Thus
\[
\supp{u(\cdot,t)}\subset \supp{\tilde{u}(\cdot,t)} = (-\infty, t^{\frac{1}{2s}} \xi_0] \times \R^{N-1}.\qedhere
\]
\end{proof}

\begin{proof}[Proof of Theorem \ref{coro:NFiniteSpeed2}]
The estimate on the growth of the support follows from applying Lemma \ref{thm:NFiniteSpeed2} for every hyperplane tangent to  $ \partial B_R(x_0)$.

We show now that there is a maximum support for $u$. Without loss of generality assume that $x_0=0$ and take
\begin{equation*}
\hat{h}_0(x):=\begin{cases}
L+\|\Phi(h_0)\|_{L^\infty(\R^N)}  \qquad\qquad&\textup{if}\qquad   x\in B_R(0)\\
L-\veps \qquad\qquad&\textup{if}\qquad   x\in B_R(0)^c,
\end{cases}
\end{equation*}
and let $\hat{h}$ be the corresponding very weak solution of \eqref{P1}--\eqref{P1-Init} and $\hat{u}=(\hat{h}-L)_+$. Clearly, $h_0\leq \hat{h}_0$ and thus, by comparison, $h\leq \hat{h}$. Thus, $\supp\{u(\cdot,t)\}\subset \supp\{\hat{u}(\cdot,t)\}$. Then it is enough to prove the estimate of maximum support for $\hat{u}$.

 Since $L-\veps$ is a stationary solution of \eqref{P1}--\eqref{P1-Init}, then $L^1$ contraction implies
 \[
 \int_{\R^N}(\hat{h}(x,t)-(L-\veps))_+\dd x\leq \int_{\R^N}(\hat{h}_0(x)-(L-\veps))_+\dd x= \left(\|\Phi(h_0)\|_{L^\infty(\R^N)}+\veps\right)|B_R|
 \]
 On the other hand, $\hat{h}(x,t)\geq L$ for $x\in \supp\{\hat{u}(\cdot,t)\}$, and then
 \[
\begin{split}
\int_{\R^N}(\hat{h}(x,t)-(L-\veps))_+\dd x&\geq\int_{\supp\{\hat{u}(\cdot,t)\}}(\hat{h}(x,t)-(L-\veps))_+\dd x\geq \veps |\supp\{\hat{u}(\cdot,t)\}|.
\end{split}
\]
Thus,
\[
|\supp\{\hat{u}(\cdot,t)\}|\leq (\veps^{-1}\|\Phi(h_0)\|_{L^\infty(\R^N)}+1)|B_R|.
\] Finally we note that $h_0$ being radially symmetric implies that $u(\cdot,t)$ is also radially symmetric (and continuous) for all $t>0$, and thus, $\supp\{\hat{u}(\cdot,t)\}= B_{\tilde{R}}(0)$ for some $\tilde{R}>0$. The relation between $\tilde{R}$ and $R$ is immediate from the estimate above.
\end{proof}

%%%%%%%%%%%%%%%%%%%%%%%%%%%%%%%%%%%%%%%%%%%%%%%%%%%%%%%%%%%%%%%%%%%%%%%%%%%%%%

\subsection{Infinite speed of propagation and tail estimates of the enthalpy}

We are also able to prove some more propagation properties of the solution. In the following result we do not need to rely on the properties of selfsimilar solutions.
Here, we  provide results on infinite speed of propagation of the enthalpy variable $h$, and give precise estimates on the tail of the solution.  For simplicity, we will state it only for positive solutions, but as usual, the result can be extended for any kind of bounded solutions due to Remark \ref{rem:translationofL}.

\begin{theorem}[Infinite speed of propagation and tail behaviour for $h$]\label{thm:infinitespeedh}
Let $0\leq h\in L^\infty(Q_T)$ be the very weak solution of \eqref{P1}--\eqref{P1-Init} with $0\leq h_0\in L^\infty(\R^N)$ as initial data.
\begin{enumerate}[{\rm (a)}]
\item If $h_0\ge L+ \veps>L$ in $B_\rho(x_1)$ for $x_1\in \R^N$ and $\rho,\veps>0$, then $h(\cdot,t)>0$ for all $t>0$.
\item If additionally $\supp\{h_0\}\subset B_\eta(x_0)$ for $x_0\in \R^N$ and $\eta>0$ , then
\begin{equation}
\label{eq:tailh}
h(x,t)\asymp 1/|x|^{N+2s} \qquad \textup{for all $t\in(0,T)$ and $|x|$ large enough.}
\end{equation}
\end{enumerate}
\end{theorem}

\begin{remark}
Note that in the proof below, we additionally prove that there exits a small time $t^*$ such that
$h$ exhibits the precise tail behaviour:
$$
h(x,t)\asymp t/|x|^{N+2s} \qquad \mbox{for all $t\in(0,t^*)$ and $|x|$ large enough.}
$$
This mimics the tail behaviour of the solutions of the linear fractional heat equation for small times, see e.g. \cite{BoSiVa17}.
\end{remark}

\begin{proof}[Proof of Theorem \ref{thm:infinitespeedh}] Assume by simplicity that $h_0\geq L+1$ in  $B_\rho(x_1)$ (the argument will work in the same way replacing $L+1$ with $L+\veps$ for any $\veps>0$). Consider the initial data
\begin{equation*}
h_{0,\rho}(x):=\begin{cases}
L+1  \qquad\qquad&\textup{if}\quad   x\in B_\rho(x_1)\\
0  \qquad\qquad&\textup{if}\quad   x\in B_\rho(x_1)^c
\end{cases}
\end{equation*}
and let $h_\rho$ be the corresponding solution of \eqref{P1}--\eqref{P1-Init}. Since $h_{0,\rho}\leq h_0$, $h_\rho\leq h$ again by comparison.

\smallskip
\noindent\textbf{1)} \emph{$h(\cdot,t)>0$ for every $0<t\leq t^*<T$.}
By continuity of $\Phi(h)$, we have that there exists a time $t^*_0<T$ such that $h(x,t)\geq L+\frac{1}{2}>0$ in $(x,t)\in B_{\frac{\rho}{2}}(x_1)\times (0,t^*_0]$. Now, consider the initial data $h_{0,\frac{\rho}{8}}$ and the solution $h_{\frac{\rho}{8}}$. Clearly, $h_{0,\frac{\rho}{8}}\leq h_{0,\rho}$, and then $h_{\frac{\rho}{8}}\leq h_\rho$. By Theorem \ref{coro:NFiniteSpeed2},
\[
\supp\{(h_{\frac{\rho}{8}}(\cdot,t)-L)_+\}\subset B_{\frac{\rho}{8}+\xi_0 t^{\frac{1}{2s}}}(x_1)
\]
for some $\xi_0>0$. Thus, we can take $t_1^*$ small enough such that
\[
\supp\{(h_{\frac{\rho}{8}}(\cdot,t)-L)_+\}\subset B_{\frac{\rho}{4}}(x_1) \quad \textup{for all}\quad  t\in (0,t_1^*],
\]
which implies that $u_{\frac{\rho}{8}}(x,t):=(h_{\frac{\rho}{8}}(x,t)-L)_+=0$ for all $(x,t)\in  B_{\frac{\rho}{4}}(x_1)^c\times (0,t_1^*]$. Thus, we can deduce, as in Lemma \ref{lem:estimHice}, the following estimate
\[
h_\frac{\rho}{8}(x,t)-h_{0,\frac{\rho}{8}}(x)=-\int_0^t \FL u(x,\tau)\dd \tau \quad \textup{for a.e.} \quad (x,t)\in  B_{\frac{\rho}{2}}(x_1)^c\times (0,t_1^*].
\]
Again, by continuity, we can assume that $h_\frac{\rho}{8}(x,t)\geq L+\frac{1}{2}>0$ in $(x,t)\in B_{\frac{\rho}{16}}(x_1)\times(0,t^*_2]$. Then, for all $(x,t)\in  B_{\frac{\rho}{2}}(x_1)^c\times (0,\min\{t_1^*,t_2^*\}]$ we have that
\begin{equation}\label{eq:lowerboundh}
\begin{split}
h_\frac{\rho}{8}(x,t)&=-\int_0^{t} \int_{\R^N} \frac{u_\frac{\rho}{8}(x,\tau)-u_\frac{\rho}{8}(z,\tau)}{|x-z|^{N+2s}}\dd z\dd \tau=\int_0^{t} \int_{\R^N} \frac{u_\frac{\rho}{8}(z,\tau)}{|x-z|^{N+2s}}\dd z\dd \tau\\
&\geq \int_0^{t} \int_{B_{\frac{\rho}{16}}(x_1)} \frac{\frac{1}{2}}{|x-z|^{N+2s}}\dd z\dd \tau=\frac{t}{2}\int_{B_{\frac{\rho}{16}}(x_1)} \frac{\dd z}{|x-z|^{N+2s}}\\
&\geq \frac{t}{2} \Big(|x-x_1|+\frac{\rho}{16}\Big)^{-N-2s}.
\end{split}
\end{equation}
The last estimate show that $
h(x,t)\geq h_\frac{\rho}{8}(x,t)>0$ for all  $(x,t)\in B_{\frac{\rho}{2}}(x_1)^c \times (0,\min\{t_1^*,t_2^*\}]$. Since  $h(x,t)>0$ for $(x,t)\in  B_{\frac{\rho}{2}}(x_1) \times (0,t^*_0]$, we conclude that $
h(x,t)>0$ for all  $(x,t)\in \R^N \times (0,\min\{t_0^*,t_1^*,t_2^*\}]$.

Note also that estimate \eqref{eq:lowerboundh} gives a quantitative estimate for $|x|$ large enough:
\[
h(x,t)\geq h_\frac{\rho}{8}(x,t) \gtrsim \frac{t}{|x|^{N+2s}} \qquad \textup{for all} \qquad t\in (0,\min\{t_0^*,t_1^*,t_2^*\}].
\]

\smallskip
\noindent\textbf{2)} \emph{$h(\cdot,t)>0$ for every $0<t<T$.}
Our strategy is to prove preservation of positivity, i.e., if $h(x_0,t_0)>0$ for some $(x_0,t_0)\in\R^N\times(0,T)$ then $h(x_0,t)>0$ for all $t\in [t_0,T)$. Actually we will prove even more: if additionally $h(x_0,t_0)\geq L$ for some $(x_0,t_0)\in\R^N\times(0,T)$, then $h(x_0,t)\geq L$ for all $t\in [t_0,T)$.

The latter is actually pretty simple for our problem: Fix a time $t^*\geq0$ and define
\[
\hat{h}(x,t):= \min\{h(x,t^*),L\}.
\]
Since $\hat{h}\leq L$, $\hat{h}$ is a stationary solution of \eqref{P1}--\eqref{P1-Init}. Note also that
\[
\hat{h}(x,t^*)\leq h(x,t^*) \qquad \textup{for a.e.} \qquad x\in \R^N.
\]
Then, by comparison,
\[
\min\{h(x,t^*),L\}=\hat{h}(x,t)\leq h(x,t) \qquad \textup{for a.e.} \qquad (x,t)\in \R^N\times[t^*,T).
\]
This fact, together with the lower bound given by \eqref{eq:lowerboundh} show that for all $t\in[\min\{t_0^*,t_1^*,t_2^*\},T)$ we have that
\[
 h(x,t)\geq \min\bigg\{L,\frac{\min\{t_0^*,t_1^*,t_2^*\}}{2} (|x-x_1|+\frac{\rho}{16})^{-N-2s}\bigg\}
\]
for a.e. $x\in B_{\frac{\rho}{2}}(x_1)^c$. In particular, this shows that the lower bound in \eqref{eq:tailh} is valid for all times and big enough $|x|$.

\smallskip
\noindent\textbf{3)} \emph{The asymptotic bounds of $h$.} The only thing left to prove is that for all $t\in(0,T)$ we have $h(x,t)\lesssim \frac{1}{|x|^{N+2s}} $ for $|x|$ large enough.  Recall that  $\supp\{h_0\}\subset B_\eta(x_0)$ and take
\[
\tilde{h}_0=\begin{cases}
\max\{L+1, \esssup_{\R^N}\{h_0\}\}  \qquad\qquad&\textup{if}\quad   x\in B_\eta(x_0)\\
0  \qquad\qquad&\textup{if}\quad   x\in B_\eta(x_0)^c
\end{cases}
\]
with $\tilde{h}$ being the corresponding solution of \eqref{P1}--\eqref{P1-Init}. Since  $h_0\leq \tilde{h}_0$ implies $h\leq \tilde{h}$, it remains to prove the upper bound for $\tilde{h}$.

By Theorem \ref{coro:NFiniteSpeed2}, there exists an $R$ big enough such that for all $t\in(0,T)$ we have that
\[
\supp\{(\tilde{h}(\cdot,t)-L)_+\}\subset B_R(x_0).
\]
Then, for a.e. $(x,t)\in B_{2R}(x_0)^c\times(0,T)$, we have the estimate (as in Step 1)
\[
\tilde{h}(x,t)=-\int_0^t \FL \tilde{u}(x,\tau)\dd \tau,
\]
where $ \tilde{u}(x,t):=(\tilde{h}(x,t)-L)_+$. Then, since $\|\tilde{h}\|_{L^\infty}\leq \|\tilde{h}_0\|_{L^\infty}$,
\[
\begin{split}
\tilde{h}(x,t)&=-\int_0^t \int_{\R^N} \frac{\tilde{u}(x,\tau)-\tilde{u}(z,\tau)}{|x-z|^{N+2s}}\dd z\dd \tau= \int_0^t \int_{B_R(x_0)} \frac{\tilde{u}(z,\tau)}{|x-z|^{N+2s}}\dd z\dd \tau\\
&\leq t \big(\max\{L+1, \textstyle{\esssup_{\R^N}}\{h_0\}\}-L\big) \int_{B_R(x_0)} \frac{\dd z}{|x-z|^{N+2s}}\\
&\leq T \big(\max\{L+1, \textstyle{\esssup_{\R^N}}\{h_0\}\}-L\big)|B_R| (|x-x_0|-R)^{-N-2s}
\end{split}
\]
for a.e. $x\in B_{2R}(x_0)^c$ and all $t\in(0,T)$. This shows the upper bound in \eqref{eq:tailh}  for all $t\in(0,T)$ and $|x|$ large enough.
\end{proof}

%%%%%%%%%%%%%%%%%%%%%%%%%%%%%%%%%%%%%%%%%%%%%%%%%%%%%%%%%%%%%%%%%%%%%%%%%%%%%
\section{Conservation of positivity for the temperature}
\label{sec.conspos}

The following  result shows that the positivity set of the temperature of any  solution does not shrink in time at any place. This is an important quantitative aspect of the theory.

\begin{theorem}[Conservation of positivity for $u$] \label{thm:conspositivityu}
Let $h$ be the very weak solution of \eqref{P1}--\eqref{P1-Init} with $ h_0\in L^\infty(\R^N)$ as initial data and let $u=(h-L)_+$ be the temperature. If $u(x,t^*)>0$ in an open set  $\Omega\subset \R^N$ for a given time $t^*\in(0,T)$, then
\[
u(x,t)>0 \qquad \textup{for all} \qquad (x,t)\in \Omega \times [t^*,T).
\]
The same result holds for $t^*=0$ if $u_0=\Phi(h_0)$ is continuous in $\Omega$.
\end{theorem}

\begin{proof}[Proof of Theorem \ref{thm:conspositivityu}]
We are going to prove it in several steps. We found it convenient to make a detour
via elliptic problems by a discretization in time in the style of Crandall-Liggett \cite{CrLi71}. We first derive conservation of positivity and comparison results in the elliptic setting and then pass to the limit from the discretized to the original problem.

\noindent\textbf{1) }\emph{Reduction.}  Under our assumptions, we have  from Theorem \ref{thmfrac-cont} that $\Phi(h(\cdot,t^*))$ is continuous in $\Omega$. Moreover: $h(x,t^*)>L$ (or $u(x,t^*)>0$) for all $x\in \Omega$ for some open set $\Omega\subset \R^d$ and some fixed time $t^*\in[0,T)$. Take any $x_1\in \Omega$ and let $R>0$ such that $B_R(x_1)\subset\subset \Omega$. Without loss of generality, we can assume that $x_1=0$, $t^*=0$ and $h(\cdot,0)\geq L+1$ in $B_R(0)$, $h(\cdot,0)\geq 0$ in $\R^N$. Now let
\[
\hat{h}_0(x)=\begin{cases}
L+1 \qquad\qquad&\textup{if}\qquad   x\in B_R(0)\\
0 \qquad\qquad&\textup{if}\qquad   x\in B_R(0)^c,
\end{cases}
\]
and $\hat{h}$ the corresponding distributional solution of  \eqref{P1}--\eqref{P1-Init}. Clearly $\hat{h}_0\leq h(\cdot,0)$, and thus $\hat{h}\leq h$. So if we prove preservation of positivity for $\hat{u}=(\hat{h}-L)_+$ in $B_R(0)$ we are done. Note that this puts us in a $L^1(\R^N)$ framework
where the Crandall-Liggett theorem applies, since the equations contain nonlinear $m$-accretive operators (see e.g. \cite{CrLi71,CrPi82}).

\smallskip
\noindent\textbf{2)} \emph{Conservation of positivity of the elliptic problem.} By time discretization of our original equation \eqref{P1} we are led to consider the nonlocal elliptic problem
\begin{equation}\label{eq:ellipticproblem}
g+\veps \FL (g-L)_+=f,
\end{equation}
where $\veps$ is the time step.  This is usually written in the elliptic literature as
$$
\veps \FL v+ \beta(v) =f,
$$
where the equivalence is done by putting $\Phi(g)=v$ and $\beta=\Phi^{-1}$ (so that $\beta$ is a  linear function for $v>0$ with a jump at $v=0$).
The theory for this equation is well-known for $s=1$ where the famous paper \cite{BBC75} proves that the problem posed in $\R^N$ with $f\in L^1(\R^N)$ has a unique solution $g\in L^1(\R^N)$ and the map $f\mapsto g$ is an $L^1(\R^N)$-contraction. Moreover, the maximum principle holds and the $L^p$ norms are conserved. The result has been extended in \cite{ChenVeron} to the fractional version with $0<s<1$  posed in bounded domains with zero outer conditions. From this theory it follows that when $f$ is radially symmetric and nonincreasing in $|x|$ (as is the case for $\hat{h}_0$), then $g$ is also radially symmetric and nonincreasing in $|x|$. Moreover, writing $\veps \FL v=f- \beta(v)\in L^1(\R^N)\cap L^\infty(\R^N)$
implies that $v$ is H\"older continuous in $B_R$.

\noindent {\bf Claim.} \emph{Under the above conditions, equation \eqref{eq:ellipticproblem} has the following conservation (including possible expansion) of the water region set:
\[
B_r(0)=\{y\in \R^N \, : \, (f(y)-L)_+>0\} \subset \{y\in \R^N \, : \, (g(y)-L)_+>0\}=B_{r'}(0)
\]
where $r\leq r'$.}

\noindent \emph{Proof of the claim.} The fact that the above sets are balls is ensured by the radial symmetry of $f$ and $g$. To prove the claim we argue by contradiction and assume $r>r'$. Note that for every $x_0\in B_r\setminus\overline{B_{r'}}$ equation \eqref{eq:ellipticproblem} is satisfied in a pointwise sense since $(g-L)_+=0$ in  $B_\eta(x_0)$ for some $\eta>0$. Since $x_0$ is a strict global minimum of $(g-L)_+$,  we get $\FL (g-L)_+(x_0)<0$,
and thus, by \eqref{eq:ellipticproblem},
\[
f(x_0)=g(x_0)+\veps \FL (g-L)_+(x_0)< g(x_0)
\]
which implies the contradiction $g(x_0) > f(x_0)>L$.

Let now $f$  stand for $\hat{h}^{k-1}$ in \eqref{eq:ellipticproblem},  the discretized enthalpy at time $t^{k-1}$, and then $g=\hat{h}^k$ is discretized enthalpy at time $t_k=t_{k-1}+\veps$, starting from $\hat{h}^0=\hat{h}_0$. In particular, for all $k\geq0$ we get
\[
B_R(0)=\{y\in \R^N \, : \, (\hat{h}_0(y)-L)_+>0\}\subset\{y\in \R^N \, : \, (\hat{h}^k(y)-L)_+>0\}
\]

\smallskip
\noindent\textbf{3)} \emph{Subsolutions.} We want to prove comparison with a nontrivial positive subsolution. Let us first construct a suitable iteration of subsolutions.

Take $R'<R$ and consider the first eigenfunction $\varphi>0$ and first eigenvalue $\lambda_1$ of $\FL$ posed in ${B_{R'}(0)}$ with zero Dirichlet outside conditions, i.e., $\FL \varphi =\lambda_1 \varphi$ in $B_{R'}(0)$ and $\varphi=0$ in $B_{R'}(0)^c$. For every  $\veps>0$, we have that
$$
\varphi+\veps \FL \varphi=(1+\veps \lambda_1) \varphi \quad \textup{in} \quad B_{R'}(0).
$$
Now, we consider the elliptic equation \eqref{eq:ellipticproblem} restricted to $B_{R'}(0)$. Then we can use the above to see that for any $c>0$ the function $\underline{g}:=c(1+\veps \lambda_1)^{-1} \varphi+L$  satisfies \eqref{eq:ellipticproblem} in $B_{R'}(0)$ with right-hand side $\underline{f}:=c\varphi+L$. We also know that $\underline{g}=L$ in $B_{R'}(0)^c$ by definition of $\varphi$. In complete picture: the iterated solution $\underline{h}^k$ of \eqref{eq:ellipticproblem} in $B_{R'}(0)$ with $\underline{h}^k=L$ in $B_{R'}(0)^c$ and $\underline{h}^0:=c\varphi+L$ is given by
$$
\underline{h}^k=c(1+\veps \lambda_1)^{-k}\varphi+L.
$$
Note that positivity of $\varphi$ in $B_{R'}(0)$ implies that $\underline{h}^k-L>0$ in $B_{R'}(0)$. Thus, $\underline{h}^k$ is a good candidate for a subsolution.

\smallskip
\noindent\textbf{4)} \emph{Comparison.} We now want to prove that for $\hat{h}^k$ and $\underline{h}^k$ defined as in Step 2 and Step 3 we have that $\hat{h}^k> \underline{h}^k$ in $B_R(0)$  for all $k\geq0$.

Given any $R'<R$, we take $c>0$ small enough such that, $c\varphi<1$ in $B_{R'}(0)$ , i.e.
\[
\hat{h}^0=L+1>L+ c\varphi =\underline{h}^0 \quad \textup{in} \quad B_{R'}(0).
\]
By Step 2 and Step 3 we also know that $\hat{h}^k>L= \underline{h}^k$ in $B_{R}(0)\setminus B_{R'}(0)$ for all $k\geq0$. Assume by induction that $f:=\hat{h}^k> \underline{h}^k=:\underline{f}$ in $B_{R'}(0)$ and we want to conclude $g:=\hat{h}^{k+1}> \underline{h}^{k+1}:=\underline{g}$ in $B_{R'}(0)$. Arguing by contradiction we can assume that there exists a point $x_1$ of nonpositive minimum for $g-\underline{g}$ in  $B_{R'}(0)$. Then $x_1$ is also a point of nonpositive minimum of $(g-L)_+-(\underline{g}-L)_+$ in $\R^N$ since $(g-L)_+\geq 0$ and $(\underline{g}-L)_+=0$ (i.e. $(g-L)_+-(\underline{g}-L)_+\geq0$) in $B_{R'}(0)^c$, and
\[
(g-L)_+-(\underline{g}-L)_+=g-L-(\underline{g}-L)=g-\underline{g}\quad  \textup{in} \quad B_{R'}(0)
\]
by Step 2 and Step 3. Thus, $\FL ((g-L)_+-(\underline{g}-L)_+)(x_1)<0$.
Then for $x_1\in B_{R'}(0)$ we have
\[
0\geq g-\underline{g}=(f-\underline{f})-\big(\FL(g-L)_+-\FL(\underline{g}-L)_+\big)>0
\]
which is a contradiction. Thus $g>\underline{g}$ in $B_{R'}(0)$, which concludes this step.

\smallskip
\noindent\textbf{5)} \emph{Limit.} We now use the previous constructions and pass to the limit in the Crandall-Liggett process. Let $\hat{h}^k$ and $\underline{h}^k$ be defined as in Step 2 and Step 3 (we add the subindex $\veps$ to emphasize that we pass to the limit as $\veps \to 0^+$). Let $\hat{h}_\veps: \R^N\times[0,\infty)\to \R$ and $\underline{h}_\veps: B_{R}(0)\times[0,\infty)\to \R$ be defined by
\[
\hat{h}_\veps(x,t):=\hat{h}^{k}_\veps(x) \qquad \textup{for all} \qquad (x,t)\in \R^N \times [k\veps, (k+1)\veps)
\]
and
\[
\underline{h}_{\veps}(x,t):=\underline{h}_{\veps}^k(x) \qquad \textup{for all} \qquad (x,t)\in B_{R}(0)\times [k\veps, (k+1)\veps).
\]
The sequences converge in $C([0,T]: L^1(\R^N))$ and $C([0,T]: L^1(B_{R}(0))) $ respectively, and thus pointwise almost everywhere.  In particular, we get for a.e. $x\in B_R(0)$ and $t\in(0,T)$ that
\[
(h(x,t)-L)_+ \geq (\hat{h}(x,t)-L)_+=\lim_{\veps \to 0^+} (\hat{h}_\veps(x,t)-L)_+ \geq \lim_{\veps \to 0^+} (\underline{h}_\veps(x,t)-L)_+=c\e^{-\lambda_1t}\varphi(x) >0.
\]
By continuity of $(h(x,t)-L)_+ $, the result holds for all $(x,t)\in B_R(0)\times(0,T)$ which concludes the proof.
\end{proof}

\begin{remark}
\rm
\begin{enumerate}[{\rm (a)}]
\item It is trivial to obtain that $h(x,t^*)\geq L$ implies that  $h(x,t)\geq L$ for all $t\in[t^*,T)$ by comparing with the stationary solution $\min\{h(x,t^*),L\}$. The above result is stronger since it provides strict inequalities.
\item The proof of Theorem \ref{thm:conspositivityu} additionally reveals the following estimate: Given any ball $B_R(x_0)\subset\Omega$, we have that $u(x,t)\geq c\textup{e}^{-\lambda_1 t}\varphi(x)$ for all $(x,t)\in B_R(x_0)\times[t^*,T)$ for some $c>0$. Here $(\lambda_1,\varphi)$ is such that $\FL\varphi=\lambda_1\varphi$ in $B_R(x_0)$  and $\varphi=0$ in $B_R(x_0)^c$ (the first eigen pair of $\FL$ in $B_R(x_0)$).
\end{enumerate}
\end{remark}

%%%%%%%%%%%%%%%%%%%%%%%%%%%%%%%%%%%%%%%%%%%%%%%%%%%%%%%%%%%%%%%%%%%%%
\section{Dependence on $L$. The extremal cases $L=0$ and $L=+\infty$}
\label{sect.limit}

In previous sections we have implicitly used solutions of two associated problems to sandwich the solution $u$ of \eqref{P1}--\eqref{P1-Init}.  Actually, we will obtain the associated problems as a limit of our problems when moving the two parameters of the equation, $L$ and $s$. We devote this section to study in detail the dependence of the problem on $L$.  For the sake of simplicity, we assume that $ h_0\in L^\infty(\R^N)$ is such that $u_0:=(h_0-L)_+\in C(\Omega)$ and $u_0>0$ in an open set $\Omega\subset\R^N$.

\begin{theorem}\label{thm:subsupersol}
Let $h\in L^\infty(Q_T)$ be the very weak solution of \eqref{P1}--\eqref{P1-Init} with $ h_0\in L^\infty(\R^N)$ as initial data, and $\Omega\subset \R^N$ be an open set such that  $u_0$ is continuous and positive  in $\Omega$. Consider the problems
\begin{equation}\label{eq:FHERN}
\begin{cases}
\dell_t\overline{u} + \FL \overline{u}=0   \qquad\qquad&\textup{in}\qquad  \mathcal{D}'( \R^N\times (0,T)),\\
\overline{u}(\cdot,0)=u_0 \qquad\qquad&\textup{in}\qquad \R^N,
\end{cases}
\end{equation}
and
\begin{equation}\label{eq:FHEdom}
\begin{cases}
\dell_t\underline{u} + \FL \underline{u}=0   \qquad\qquad&\textup{in}\qquad  \mathcal{D}'(\Omega \times (0,T)),\\
\underline{u}=0    \qquad\qquad&\textup{in}\qquad  \Omega^c \times [0,T),\\
\underline{u}(\cdot,0)= u_0 \qquad\qquad&\textup{in}\qquad  \Omega.
\end{cases}
\end{equation}
Then $\underline{u}\leq u \leq \overline{u}$ in $\R^N\times(0,T)$.
\end{theorem}

\begin{proof}
By Theorem \ref{thm:conspositivityu} we have that $u(x,t)>0$ (and thus $h(x,t)=u(x,t)+L$) for all $(x,t)\in\Omega\times[0,T)$.  Recall also that $u=(h-L)_+\geq0$ in $\R^N\times(0,T)$. Then $u$ satisfies
\[
\begin{cases}
\dell_tu + \FL u=0 \qquad\qquad&\textup{in}\qquad  \mathcal{D}'(\Omega \times (0,T)),\\
u\geq 0  \qquad\qquad&\textup{in}\qquad  \Omega^c \times (0,T).
\end{cases}
\]
Both $u$ and $\underline{u}$ solve the fractional heat equation in $\Omega \times (0,T)$.  Moreover, $u(x,0)=u_0(x)\geq u_0(x) \indik_{\Omega}(x)=\underline{u}(x,0)$ in $\R^N$ and $u\geq \underline{u}$ in $\Omega^c \times (0,T)$, then standard comparison shows that $u\geq \underline{u}$ in $\R^N\times [0,T)$ (see \cite{FeRo16}).

For the upper bound, we consider $\overline{h}_0(x):=\max\{h_0(x),L\}$ and let $\overline{h}$ be the corresponding very weak solution of  \eqref{P1}--\eqref{P1-Init}. Since $\overline{h}_0\geq h_0$ and $\overline{h}_0\geq L$, the comparison principle gives $\overline{h}\geq h$ and $\overline{h}\geq L$. Then  $\overline{u}:=\Phi(\overline{h})=(\overline{h}-L)_+=\overline{h}-L$ and $\overline{u}(\cdot,0)=u_0$. This implies that
\[
u(x,t)=(h(x,t)-L)_+\leq (\overline{h}(x,t)-L)_+=\overline{u}
\]
and that $\overline{u}$ satisfies the fractional heat equation in $\R^N\times(0,T)$.
\end{proof}

In fact $\overline{u}$ and $\underline{u}$ are obtained as the corresponding limits as $L\to 0^+$ and $L\to +\infty$ respectively when $\Omega$ is considered to be maximal open set where $u_0$ is positive.

\begin{theorem}\label{thm:subsupersollim}
Assume $0\leq u_0\in L^\infty(\R^N)$.  Let $\Omega$ be the biggest open set such that $u_0(x)>0$ for all $x\in \Omega$  and assume that $u_0\in C(\Omega)$. Define the initial data
\[
h_{0,L}=\begin{cases}
L+u_0(x) \quad&\textup{in}\quad  \Omega, \\
0  \qquad&\textup{in}\quad   \Omega^c,
\end{cases}
\]
and let $h_L\in L^\infty(\R^N)$ be the corresponding very weak solution of \eqref{P1}--\eqref{P1-Init} with $u_L:=(h_L-L)_+$. Then, with $\overline{u}$ and $\underline{u}$ as in Theorem \ref{thm:subsupersol}, we have:
\begin{enumerate}[\rm (a)]
\item\label{thm:limits-item-a} $\{u_L\}_{L>0}$ is monotone in $L$, that is, if $L_2\geq L_1$ then $u_{L_2}\leq u_{L_1}$.
\item\label{thm:limits-item-b} $u_L  \to \overline{u} $ pointwise in $Q_T$ as $L\to 0^+$.
\item\label{thm:limits-item-c} $u_L  \to \underline{u} $ pointwise in $Q_T$ as $L\to +\infty$.
\end{enumerate}
\end{theorem}

\begin{proof}
By Remark \ref{rem:translationofL} we can assume that the latent heat is 0 by shifting the problem, i.e., $u_L=\Phi(h_L)=(h_L)_+$, \[
h_{0,L}(x)=\begin{cases}
u_0(x) \qquad\qquad&\textup{in}\qquad  \Omega, \\
-L  \qquad\qquad&\textup{in}\qquad   \Omega^c,
\end{cases}
\]
and $h_L$ solves
\[
\partial_t h_L +\FL (h_L)_+=0 \qquad\qquad\textup{in}\qquad \mathcal{D}'(\R^N\times(0,T)).
\]

\smallskip

\noindent\eqref{thm:limits-item-a} Now,  $h_{0,L_2}\leq h_{0,L_1}$ if $L_2\geq L_1$. By comparison $h_{L_2}\leq h_{L_1}$ and thus $u_{L_2}\leq u_{L_1}$.

\smallskip
\noindent\eqref{thm:limits-item-b} By Theorem \ref{thm:subsupersol} we know that $h_L\leq (h_L)_+=u_L \leq \overline{u}$
and thus $\{h_L\}_{L>0}$ bounded from above uniformly in $L$. This fact, together with the monotonicity proved in \eqref{thm:limits-item-a} ensures the existence of the pointwise limit
\[
w(x,t):=\lim_{L\to 0^+}{ h_L(x,t)}.
\]
Note also that $u_0(x)= \lim_{L\to 0^+}{ h_{0,L}(x)}$. It is now standard now to check that for any $\psi\in C_\textup{c}^\infty(\R^N \times [0,T))$ we have that
\[
\begin{split}
0=&\int_0^T \int_{\R^N} \big(h_L \partial_t \psi - (h_L)_+\FL \psi\big)\dd x \dd t + \int_{\R^N} h_{0,L}(x) \psi(x,0)\dd x\\
&\stackrel{L\to 0^+}{\longrightarrow} \int_0^T \int_{\R^N} \big(w \partial_t \psi - w_+\FL \psi\big)\dd x \dd t + \int_{\R^N} u_0(x) \psi(x,0)\dd x.
\end{split}
\]
Thus $w$ is a very weak solution of \eqref{P1}--\eqref{P1-Init} with initial data $u_0\geq0$. By the comparison principle we get that $w\geq0$, and so $w=w_+$. Thus, $w$ solves \eqref{eq:FHERN}, which by uniqueness (cf. \cite{BoSiVa17})  implies that $w=\overline{u}$.

\smallskip
\noindent\eqref{thm:limits-item-c} Since $\{u_L\}_{L>0}$ is uniformly bounded from below by $\underline{u}$, and monotone by part \eqref{thm:limits-item-a}, we can define $v(x,t):=\lim_{L\to +\infty} u_L(x,t)$.

First, we show that $v= 0$ in $\Omega^c\times[0,T)$. On one hand, $v(x,0)=\lim_{L\to \infty}{ u_L(x,0)}=u_0(x)$ and thus $v(\cdot,0)=0$ in $\Omega^c$.  On the other hand, assume by contradiction that there exists a set $\Omega^*\subset \Omega^c$ such that $|\Omega^*|>0$ and $v>0$ in $\Omega^*$. Since $\{u_{L}\}_{L>0}$ is nonincreasing as $L\to+\infty$,  $u_{L}>0$ in $\Omega^*$, and so $h_L=u_L>0$ in $\Omega^*$. We now use the  definition of very weak solutions (see Definition \ref{def:distSolfrac} and Lemma \ref{lem:ConservationOfMass})  with the test function $\psi(x,t):=\varphi(x)\indik_{[0,\tau]}(t)$ where $\varphi\ge 0$, smooth, and supported in $\Omega^*$. We get
\[
\int_{\Omega^*} (h_{L}(x,\tau)-h_{0,L}(x)) \varphi(x) \dd x=- \int_{0}^\tau \int_{\R} u_L(x,t) \FL \varphi(x) \dd x \dd t.
\]
The right hand side is bounded by $\tau \|u_0\|_{L^\infty(\R^N)} \|\FL \varphi\|_{L^1(\R^N)}$ which is uniformly bounded in $L$, while for the right hand side we have
\[
\int_{\Omega^*} (h_{L}(x,\tau)-h_{0,L}(x)) \varphi(x) \dd x > L \int_{\Omega^*}\varphi(x)\dd x \to+\infty \qquad \textup{as} \qquad L\to +\infty.
\]
This is a contradiction which shows that $|\Omega^*|=0$. Since $\tau>0$ was arbitrary, we have shown that $v= 0$ in $\Omega^c\times[0,T)$.

Second, the strict positivity of $\underline{u}$ in $\Omega$ (cf. \cite{FeRo16}) implies  (by Theorem \ref{thm:subsupersol}) that for all $t\in(0,T)$, we have that
\[
u_L(x,t) \geq \underline{u}(x,t)>0 \qquad \textup{for all} \qquad  x\in\Omega,
\]
that is, $h_L=u_L$ in $\Omega$. Then, for any $\psi \in C_\textup{c}^\infty(\Omega\times[0,T))$ we have that,
\[
\begin{split}
0=&\int_0^T \int_{\R^N} h_L \partial_t \psi - \int_0^T \int_{\R^N} (h_L)_+\FL \psi\dd x \dd t + \int_{\R^N} h_{0,L}(x) \psi(x,0)\dd x\\
=& \int_0^T \int_{\Omega } u_L \partial_t \psi - \int_0^T \int_{\R^N} u_L \FL \psi\dd x \dd t + \int_{\Omega} u_{0}(x) \psi(x,0)\dd x\\
& \stackrel{L\to +\infty}{\longrightarrow}  \int_0^T \int_{\Omega } v \partial_t \psi - \int_0^T \int_{\R^N} v \FL \psi\dd x \dd t + \int_{\Omega} u_0(x) \psi(x,0)\dd x\\
=& \int_0^T \int_{\Omega } v \partial_t \psi - \int_0^T \int_{\Omega} v \FL \psi\dd x \dd t + \int_{\Omega} u_0(x) \psi(x,0)\dd x.
\end{split}
\]
Thus, $v$ solves \eqref{eq:FHEdom}, which by uniqueness (cf. \cite{FeRo16}) implies that $v=\underline{u}$.
\end{proof}
\subsection{Limits in the selfsimilar case}

We can formulate a version of Theorem \ref{thm:subsupersol} and Theorem \ref{thm:subsupersollim} at the level of the selfsimilar solutions presented in Section \ref{sec.sss}. Since the parameter $L$ has been shown to play no role in the behaviour of the free boundary, for simplicity, we fix $L=0$ (see Remarks \ref{rem:translationofL} and \ref{rem:nodeponL}).

\begin{theorem}\label{thm:limits-selfsimilar}
Assume $L=0$  and $P_1,P_2>0$. Consider the initial data
\begin{equation*}%\label{eq:HeavySide1D}
h_{0,P_2}(x):=\begin{cases}
P_1  \qquad\qquad&\textup{if}\qquad   x\leq0\\
-P_2  \qquad\qquad&\textup{if}\qquad   x>0.
\end{cases}
\end{equation*}
Let $h_{P_2}\in L^\infty(Q_T)$ be the corresponding very weak solution of \eqref{P1}--\eqref{P1-Init} and let $\xi_{0,P_2}$ be given by Theorem \ref{thm:SS-all}\eqref{thm-SS-item3}. Then $\{\xi_{0,P_2}\}_{P_2>0}$ is a monotone sequence such that
\begin{equation}\label{eq:limitsFreeB}
\xi_{0,P_2} \to 0 \quad \textup{as} \quad P_2\to+\infty \qquad \textup{and} \qquad \xi_{0,P_2} \to +\infty \quad \textup{as} \quad P_2\to0^+.
\end{equation}
The corresponding profiles $\{U_{P_2}\}_{P_2>0}$ are also monotone for every $\xi\in \R$, and the limits correspond to the solutions of the respective limit problems, i.e., $\overline{U}(x):=\overline{u}(x,1)$ and $\underline{U}:=\underline{u}(x,1)$ with $\overline{u}$ and $\underline{u}$ as in Theorem \ref{thm:subsupersol}.
\end{theorem}

\begin{remark}
In the proof of Theorem \ref{thm:limits-selfsimilar} we will make use of the results in Theorem \ref{thm:subsupersollim}. Note that we are allowed to do this by making use of the usual trick of shifting, i.e., considering the equivalent problem with $L:=P_2$ and $\hat{h}_{0,P_2}=h_{0,P_2}+P_2$.  Note also that $u_0=P_1 \indik_{\{x\leq0\}}$ is continuous in $\Omega=(-\infty,0)$, the biggest open set of strict positivity of $u_0$.
\end{remark}

\begin{proof} [Proof of Theorem \ref{thm:limits-selfsimilar}]
By Theorem \ref{thm:subsupersollim}\eqref{thm:limits-item-a} and the fact that $U_{P_2}(\xi)=(h_{P_2}(\xi,1))_+=u_{P_2}(\xi,1)$, we know that $\{U_{P_2}\}_{P_2>0}$ is a monotone sequence of functions. Since, moreover, the profiles $U_{P_2}(\xi)$ are nonincreasing in $\xi$ for all $P_2>0$, $\{\xi_{0,P_2}\}_{P_2>0}$ is monotone.

We have to establish both limits in \eqref{eq:limitsFreeB}. Note that the biggest open set where $u_{0,P_2}:=(h_{0,P_2})_+>0$ is $\Omega=(-\infty,0)$. By Theorem \ref{thm:subsupersollim}\eqref{thm:limits-item-b}, we have that
 \[
 U_{P_2}(\xi)=u_{P_2}(\xi,1)\to \overline{u}(\xi,1) \qquad \textup{as} \qquad P_2\to0^+,
 \]
 where $\overline{u}$ is the solution of \eqref{eq:FHERN}. The strict positivity of $\overline{u}$ in $\R$ automatically implies that $\xi_{0,P_2} \to +\infty$. By Theorem \ref{thm:subsupersollim}\eqref{thm:limits-item-c}, we have that
\[
U_{P_2}(\xi)=u_{P_2}(\xi,1)\to \underline{u}(\xi,1) \qquad \textup{as} \qquad P_2\to+\infty.
\]
where $\underline{u}$  is the solution of \eqref{eq:FHEdom} with $\Omega=(-\infty,0)$. Since $\underline{u}(\cdot,1)=0$ in $[0,+\infty)$ we have that $\lim_{P_2\to+\infty} \xi_{0,P_2}=0$.
\end{proof}

In Figure \ref{fig:Movinginterface1} we have computed the value of $\xi_0$ depending on $P_2$. This is done using the numerical schemes presented in Section \ref{sec:Num}. Without loss of generality we have fixed $L=0$ and $P_1=1$. The tendencies $\xi_0\to 0$ as $P_2\to +\infty$ and $\xi_0\to +\infty$ as $P_2\to 0^+$, as well as the monotonicity in $P_2$  predicted by Theorem \ref{thm:limits-selfsimilar} are clearly observed.

\begin{figure}[h!]
\includegraphics[width=0.49\textwidth]{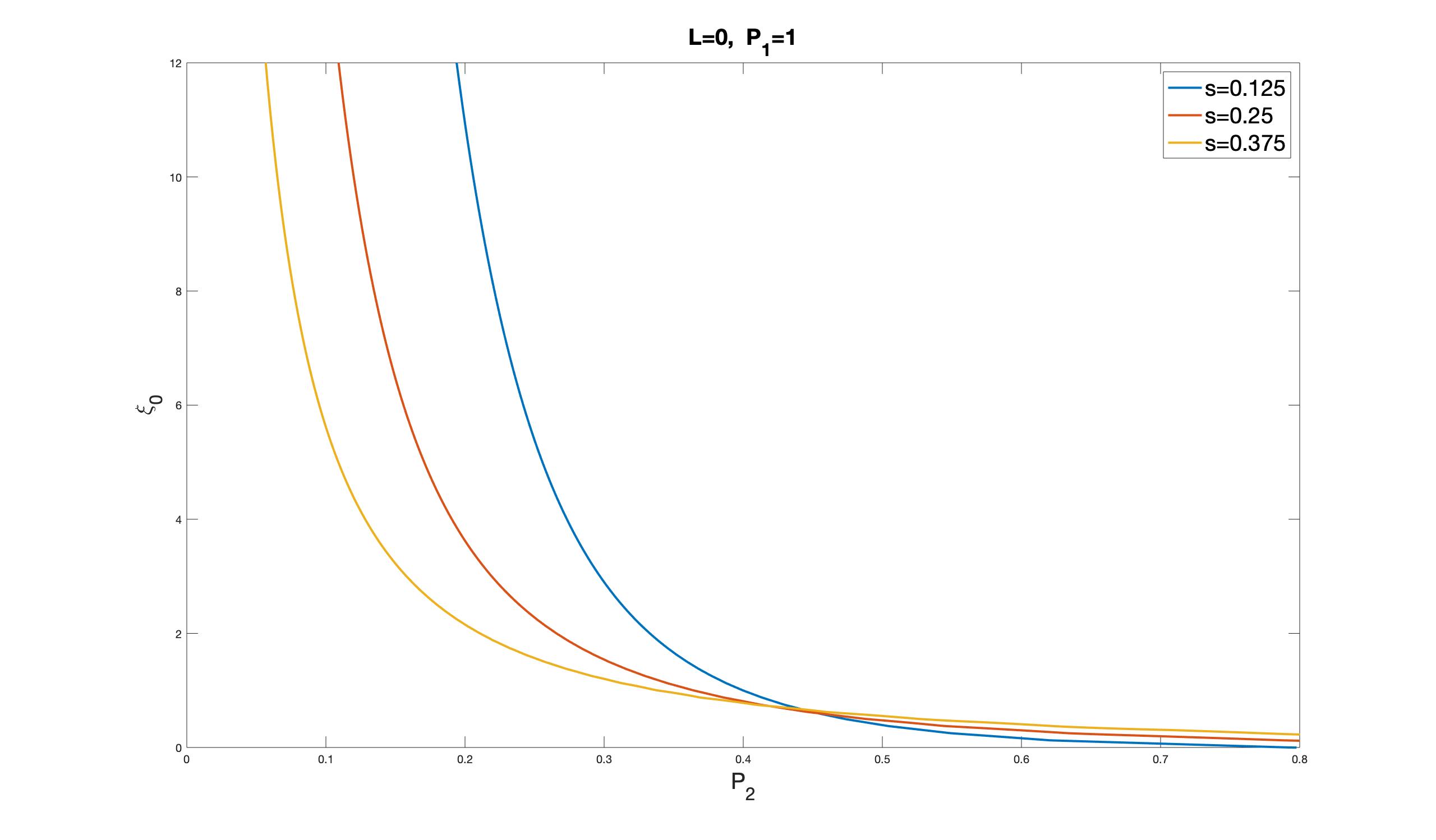}
\includegraphics[width=0.49\textwidth]{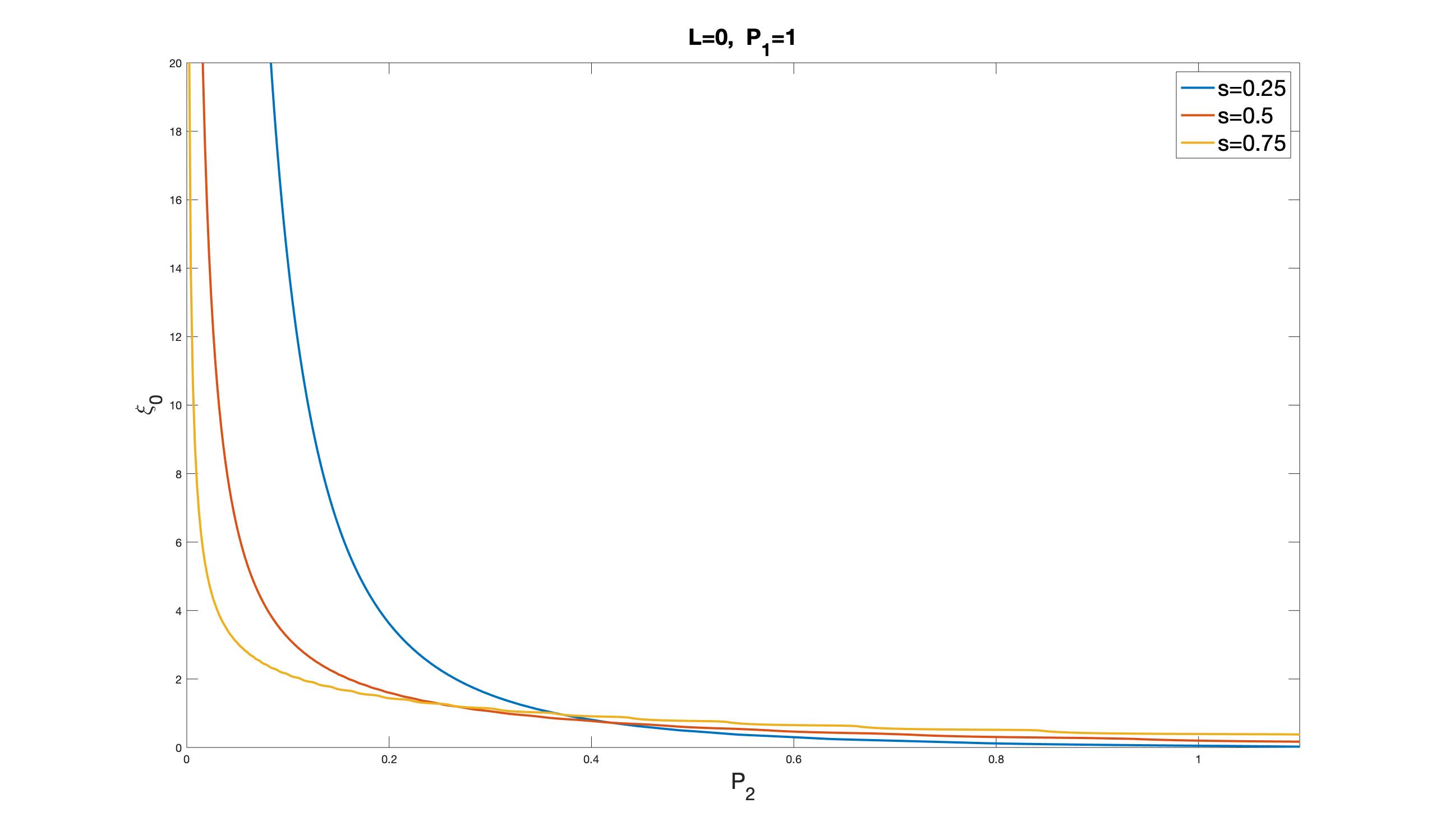}
\caption{$\xi_0$ depending on $P_2$ for a wide range of $s$.}
\label{fig:Movinginterface1}
\end{figure}

%%%%%%%%%%%%%%%%%%%%%%%%%%%%%%%%%%%%%%%%%%%%%%%%%%%%%%%%%%%%%%%%%%%%%%%%%%%%%%
\section{Dependence on $s$. The extremal cases $s=0$ and $s=1$}
\label{sect.limit-s}
In terms of the fractional parameter $s$, there are two extremal cases. On one hand, when $s\to1^-$ we recover the classical one-phase Stefan problem
\begin{equation}\label{eq:CSP}
\dell_th-\Delta \Phi(h)=0.
\end{equation}
This problem has been extensively studied as mentioned in the introduction.
On the other hand, when $s\to0^+$, we obtain
\begin{equation}\label{eq:Limit0}
\dell_th+\Phi(h)=0.
\end{equation}
This is an ODE for every fixed $x$, whose solution is given by
\[
h(x,t)=\begin{cases}
 (h_0(x)-L)\textup{e}^{-t} +L \quad & \textup{if} \quad h_0(x)>L\\
 h_0(x)  \quad & \textup{if} \quad h_0(x)\leq L
\end{cases}
\]
and thus $u(x,t):=u_0(x)\textup{e}^{-t}$. In this extreme case, the free boundary does not move, since the positivity set of $u(\cdot,t)$ is precisely the same as the one of $u_0(x)$.
\subsection*{Convergence of solutions.}  The stability of \eqref{P1} in $s$ has been proved in \cite{ACJ2014}, where its Theorem 3 shows  that the solution of \eqref{P1} converges to the solution of \eqref{eq:CSP} as $s\to1^-$ and to the solution of \eqref{eq:Limit0} as $s\to 0^+$ as expected. The proof applies for more general equations and diffusion nonlinearities $\Phi$. See also \cite{DPQuRoVa12} and \cite{DTEnJa17a}.
\subsection*{Convergence of the selfsimilar interface.} Given $s\in(0,1)$, define $\xi_{0,s}$ to be the interface point given by Theorem \ref{thm:subsupersollim}\eqref{thm:limits-item-c}. Clearly, one expects that
\[
\xi_{0,s}\to 0 \quad \textup{as} \quad s\to0^+ \qquad \textup{and} \qquad \xi_{0,s}\to \hat{\xi}_0 \quad \textup{as} \quad s\to1^-
\]
where $\hat{\xi}_0>0$ is the interface point of the selfsimilar solutions in the classical Stefan problem.

In Figure \ref{fig:Movinginterface3} we present the above described phenomena by computing numerically (using the convergent schemes of Section \ref{sec:Num}) the value of $\xi_{0,s}$ in the range $s\in(0,1)$ for different values of $P_2$.

\begin{figure}[h!]
\includegraphics[width=\textwidth]{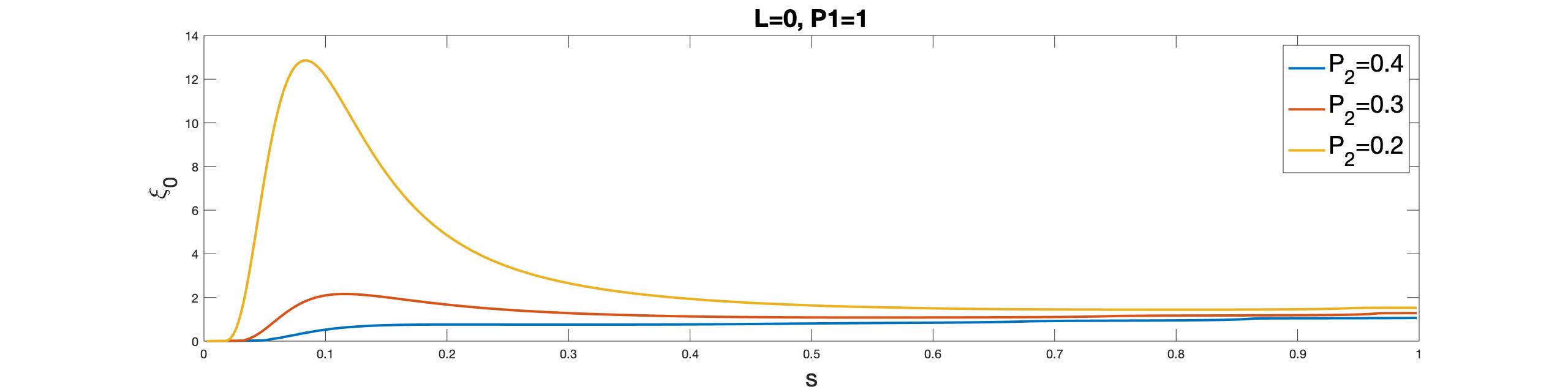}
\caption{$\xi_0$ depending on $s$ for some $P_2$.}
\label{fig:Movinginterface3}
\end{figure}

We can see several interesting phenomena in Figure \ref{fig:Movinginterface3}:

$\bullet$ The sequence $\{\xi_{0,s}\}_{s\in(0,1)}$ is not necessary monotone in $s$ and its behaviour strongly depends on $P_2$.

$\bullet$  Observe that the expected behaviour $\xi_{0,s}\to 0$ as $s\to 0^+$ is obtained independently of $P_2$.

$\bullet$ For values of $s$ close to 0, the  value of free boundary point $\xi_{0,s} $ is very sensitive to variations of $P_2$.

$\bullet$ The limit of $\xi_{0,s}$ as $s\to 1^-$ corresponds to the free boundary parameter $\hat{\xi}_0$ of the local one-phase Stefan problem. This fact can be clearly observed in Figure \ref{fig:closeto1}, where one can see that for $s=0.99$ and $s=1$ the free boundary points are essentially the same. The simulation for the local case is done using the standard finite difference method for the Laplacian.

\begin{figure}[h!]
\includegraphics[width=\textwidth]{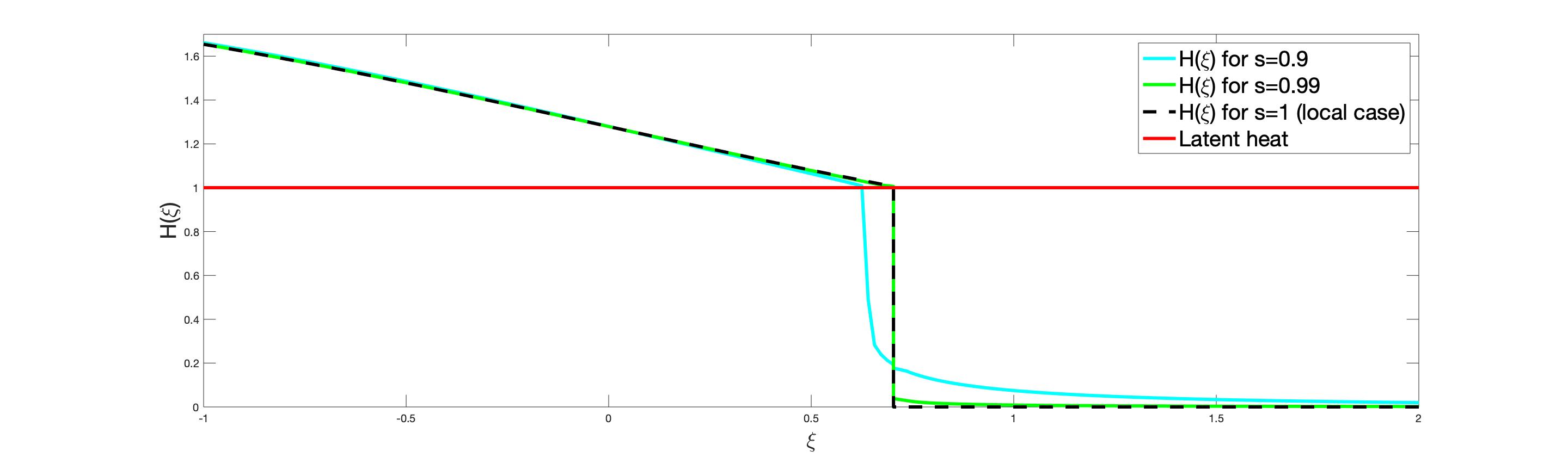}
\caption{Selfsimilar profile for values of $s$ close to 1. We fixed $L=P_1=P_2=1$.}
\label{fig:closeto1}
\end{figure}

$\bullet$ For $s$ large, the sensitivity of $\xi_0$ depending on $P_2$ is very weak.

$\bullet$ It is worth mentioning that $\xi_{0,s}$ may be larger, even much larger, or smaller than $\hat{\xi}_0$ corresponding to the classical Stefan problem, depending on the ration $P_2/P_1$. See also comment after Figure \ref{fig:SSSmovings}.

Much work is needed to understand the behaviour of the free boundaries of general solutions and their limits as $s$ tends to the limits $0^+$ and $1^-$.

\section{Asymptotic behaviour}
\label{sect.ab}

 We are also able to establish large time asymptotic behaviour for solutions with initial data with the right properties of growth/decay at $\pm \infty$. In this case, the proof is a simple consequence of the study of the selfsimilar solutions and we include it here.

\begin{theorem}[Asymptotic behaviour]\label{thm:asymBeha}
Let $\tilde{h}$ be the selfsimilar solution of \eqref{P1}--\eqref{P1-Init} for $N\ge 1$ with $\tilde{h}_0$ given by \eqref{eq:HeavySide1roratedtranslated} as initial data. Let $h_0\in L^\infty(\R^N)$ be such that $\supp\{h_0-\tilde{h}_0\}\subset \mathcal{C}_R$ where $\mathcal{C}_R$ is any cylinder with radius $ R>0$ and $\min \{\tilde{h}_0\}\leq h_0\leq \max \{\tilde{h}_0\}$. Then the corresponding very weak solution $h\in L^\infty(\R^N)$ of  \eqref{P1}--\eqref{P1-Init} satisfies
\begin{equation}\label{eq:asyminf}
\|h(\cdot,t)-\tilde{h}(\cdot,t)\|_{L^\infty(\R^N)}\leq \Lambda_\infty({t^{-\frac{1}{2s}}})
\end{equation}
for some modulus of continuity $\Lambda_\infty$. Actually, for every $p>1$, we have
\[
\|h(\cdot,t)-\tilde{h}(\cdot,t)\|_{L^p(\R^N)}\leq \Lambda_p({t^{-\frac{1}{2s}}})
\]
where $\Lambda_p(\lambda)=\Lambda_\infty(\lambda)^{\frac{p-1}{p}}  (2R(P1+P2))^{\frac{1}{p}}$.
\end{theorem}

\begin{remark}
The assumptions on $h_0$ imply that it lies between two translations of Heavyside type initial data. Note, moreover, that a canonical example of $\mathcal{C}_R$ is $(-R,R)\times \R^{N-1}$.
%The assumption \color{red}$\supp\{h_0-\tilde{h}_0\}\subset \mathcal{C}_R$ for some $R>0$ means that $h_0$ differs from $\tilde{h}_0$ only in the cylinder $R$.
\end{remark}

\begin{proof}[Proof of Theorem \ref{thm:asymBeha}]
By rotational and translational invariance and the structure of our selfsimilar solutions, we can assume that we are in dimension $N=1$, that $\tilde{h}_0$  and $\tilde{h}$ are given by Theorem \ref{thm:SS-all} and also that $\supp\{h_0-\tilde{h}_0\}\subset (-R,R)$.

Clearly, we then have that $\tilde{h}_0(x+R) \leq h_0(x) \leq \tilde{h}_0(x-R)$ and  $\tilde{h}_0(x+R) \leq \tilde{h}_0(x) \leq \tilde{h}_0(x-R)$, which by comparison implies
\[
\tilde{h}(x+R,t) \leq h(x,t) \leq \tilde{h}(x-R,t)\quad \textup{and} \quad\tilde{h}(x+R,t) \leq \tilde{h}(x,t) \leq \tilde{h}(x-R,t).
\]
Recall also that $\tilde{h}(x,t)=\tilde{H}(x t^{-1/(2s)})$ where $\tilde{H}$ is given by Theorem  \ref{thm:SS-all}. Since $\tilde{H}\in C(\R)$,  $\lim_{\xi\to-\infty}\tilde{H}(\xi)=L+P_1$ and $\lim_{\xi\to+\infty}\tilde{H}(\xi)=L-P_2$, $\tilde{H}$ is in fact uniformly continuous in $\R$, that is, $|\tilde{H}(\xi_1)-\tilde{H}(\xi_2)|\leq \Lambda (|\xi_1-\xi_2|)$
for some modulus of continuity $\Lambda$ that does not depend on $\xi_1$ and $\xi_2$. We thus have that
\begin{equation*}
\begin{split}
|h(x,t)-\tilde{h}(x,t)|&\leq \tilde{h}(x-R,t) - \tilde{h}(x+R,t)= \tilde{H}((x-R)t^{-\frac{1}{2s}})-\tilde{H}((x+R)t^{-\frac{1}{2s}})\\
&\leq \Lambda (2Rt^{-\frac{1}{2s}})=: \Lambda_\infty(t^{-\frac{1}{2s}}).
\end{split}
\end{equation*}
Since the last estimate is independent of $x$ we conclude \eqref{eq:asyminf}.

For the general $L^p$-asymptotic behaviour, we get that
\[
\begin{split}
\|h(x,t)-\tilde{h}(x,t)\|_{L^p(\R)}
&=\left(\int_{\R} |h(x,t)-\tilde{h}(x,t)|^p\dd x\right)^{\frac{1}{p}}\\
&\leq \|h(\cdot,t)-\tilde{h}(\cdot,t)\|_{L^\infty(\R)}^{\frac{p-1}{p}} \left(\int_{\R} |h(x,t)-\tilde{h}(x,t)|\dd x\right)^{\frac{1}{p}}\\
&\leq \Lambda_\infty(t^{-\frac{1}{2s}})^{\frac{p-1}{p}}  \left(\int_{\R} |\tilde{h}(x-R,t)-\tilde{h}(x+R,t)|\dd x\right)^{\frac{1}{p}}\\
&\leq \Lambda_\infty(t^{-\frac{1}{2s}})^{\frac{p-1}{p}}  \left(\int_{\R} |\tilde{h}_0(x-R)-\tilde{h}_0(x+R)|\dd x\right)^{\frac{1}{p}}\\
&= \Lambda_\infty(t^{-\frac{1}{2s}})^{\frac{p-1}{p}}  (2R(P1+P2))^{\frac{1}{p}}
\end{split}
\]
The proof is complete.
\end{proof}

\section{Numerical schemes and experiments}\label{sec:Num}
In this section we first propose convergent numerical schemes and then present a number of interesting numerical experiments.

\subsection{Theoretical background}\label{sec:Num1}

The theory to solve \eqref{P1}--\eqref{P1-Init} (and more general problems like \eqref{eq:genproblem}) numerically via monotone finite-difference schemes is developed in \cite{DTEnJa18b, DTEnJa19} for initial data in $L^1\cap L^\infty$. Our selfsimilar solution $h$ comes from bounded initial data $h_0$ with the property $h_0-h_0(\cdot+\xi)\in L^1(\R^N)$ for all $\xi>0$. As we will show in Theorem \ref{thm:NumApproxInBounded}, this assumption is enough to make the framework of \cite{DTEnJa18b, DTEnJa19} still hold.

We discretize \eqref{P1}--\eqref{P1-Init} explicitly in space and time to obtain
\begin{equation}\label{P1:NumSch}
V_\beta^j=V_\beta^{j-1}-\varDelta t\Levy^{\varDelta x}\Phi(V_{\cdot}^{j-1})_\beta
\end{equation}
where $V$ is the numerical solution associated to the enthalpy, i.e $V_\beta^j\approx h(x_\beta,t_j)$ for all $x_\beta:=\beta\varDelta x  \in \varDelta x \Z^N$ and $t_j=j \varDelta t \in (\varDelta t\N)\cap[0,T]$, which for any $\varDelta x, \varDelta t>0$ defines a uniform grid. As initial condition, take either $V_\beta^0=\frac{1}{\varDelta x^N}\int_{x_\beta + \varDelta x(-1/2,1/2]^N}h_0(x)\dd x$ or just  $V_\beta^0=h_0(x_\beta)$ if $h_0$ has pointwise values.

In \eqref{P1:NumSch}, $\mathcal{L}^{\Delta x}$ is a monotone finite-difference discretization of $\FL$ which has been extensively studied in \cite{DTEnJa18b}. It takes the form:
\begin{equation}\label{eq:NumSchOperator}
\mathcal{L}^{\varDelta x}\psi(x_\beta)=\mathcal{L}^{\varDelta x}\psi_\beta=\sum_{\gamma\neq0}\big(\psi(x_\beta)-\psi(x_\beta+z_\gamma)\big)\omega_{\gamma,\varDelta x}
\end{equation}
where $\omega_{\gamma,\Delta x}=\omega_{-\gamma,\Delta x}$ are nonnegative weights chosen such that
\begin{equation}\label{eq:NumSchConsistency}
\|\mathcal{L}^{\varDelta x}\psi-\FL\psi\|_{L^1(\R^N)}\to 0 \qquad \textup{as} \qquad \varDelta x \to 0^+,
\end{equation}
for all $\psi\in C_\textup{c}^\infty(\R^N)$. Existence, uniqueness and stability properties of \eqref{P1:NumSch} can be obtained from \cite{DTEnJa19}.

\begin{remark}
Among the many possible choices, the experiments performed in Section \ref{sec:Num2} (and elsewhere in the paper)  are done using the so-called \em powers of the discrete Laplacian, \em which are known to produce errors of order $O(\varDelta x^2)$ in \eqref{eq:NumSchConsistency} independently of $s\in(0,1)$ (see Section 4.5 in \cite{DTEnJa18b}).
\end{remark}

By \cite{DTEnJa19}, we can easily deduce the following convergence result which holds even in our general context of very weak solutions.

\begin{theorem}\label{thm:NumApproxInBounded} Let $h\in L^\infty(Q_T)$ be the very weak solution of \eqref{P1}--\eqref{P1-Init} with $ h_0\in L^\infty(\R^N)$ as initial data such that $h_0-h_0(\cdot+\xi)\in L^1(\R^N)$ for all  $\xi>0$, $\varDelta t,\varDelta x>0$ be such that $\varDelta t\lesssim  \varDelta x^{2s}$, $\mathcal{L}^{\varDelta x}$ be such that \eqref{eq:NumSchOperator} and \eqref{eq:NumSchConsistency} hold, and $V_\beta^j$ be the solution of \eqref{P1:NumSch}. Then, for all compact sets $K\subset \R^N$, we have that
$$
\max_{t_j\in(\Delta t\Z)\cap[0,T]}\bigg\{\sum_{x_\beta\in (\Delta x\Z^N)\cap K}\int_{x_\beta+\Delta x(-\frac{1}{2},\frac{1}{2}]^N}|V_\beta^j-h(x,t_j)|\dd x\bigg\}\to0\qquad\textup{as}\qquad \Delta x\to0^+.
$$

\end{theorem}
The above convergence is the discrete version of convergence in $C([0,T];L_\textup{loc}^1(\R^N))$.

\begin{remark}
The condition $\varDelta t\leq C \varDelta x^{2s}$ is the so-called CFL-condition which ensures good properties of the explicit scheme, such us stability and the comparison principle. The precise constant $C$ is given in \cite{DTEnJa19} and depends on $s$, $\Phi$, and the $L^\infty$-norm of the initial data. Such a condition can be removed by considering an implicit scheme, but this is out of the scope of our paper.
\end{remark}

\begin{proof}[Comments on the proof of Theorem \ref{thm:NumApproxInBounded}]
A close inspection of the convergence proofs in \cite{DTEnJa19} reveals that $V_\beta^j$ converges in the above sense to the unique very weak solution $h$ if proper equicontinuity in both space and time of $V_\beta^j$ hold. These properties hold precisely when, in addition to $L^\infty$-properties, we require $\sum_\beta|V^j(x_\beta)-V^j(x_\beta+\xi)|\leq \sum_\beta|V^0(x_\beta)-V^0(x_\beta+\xi)|<\infty.$
This property is ensured by just assuming $h_0-h_0(\cdot+\xi)\in L^1(\R^N)$ for all  $\xi>0$ rather than the stronger assumption $h_0\in L^1(\R^N)$.  We also remark that the uniqueness result of \cite{GrMuPu19} is crucial to obtain convergence of the full sequence of numerical solutions (rather than up to a subsequence).
\end{proof}

\subsection{Numerical experiments}\label{sec:Num2}
Using the scheme introduced in Section \ref{sec:Num1}, we perform now a series of numerical experiments which illustrate some interesting phenomena of the one-phase fractional Stefan problem.

\subsubsection{Behaviour of selfsimilar solutions depending on $s$}

First we include a numerical simulation of the selfsimilar solutions of Theorem \ref{thm:SS-all}.

\begin{figure}[h!]
\includegraphics[width=\textwidth]{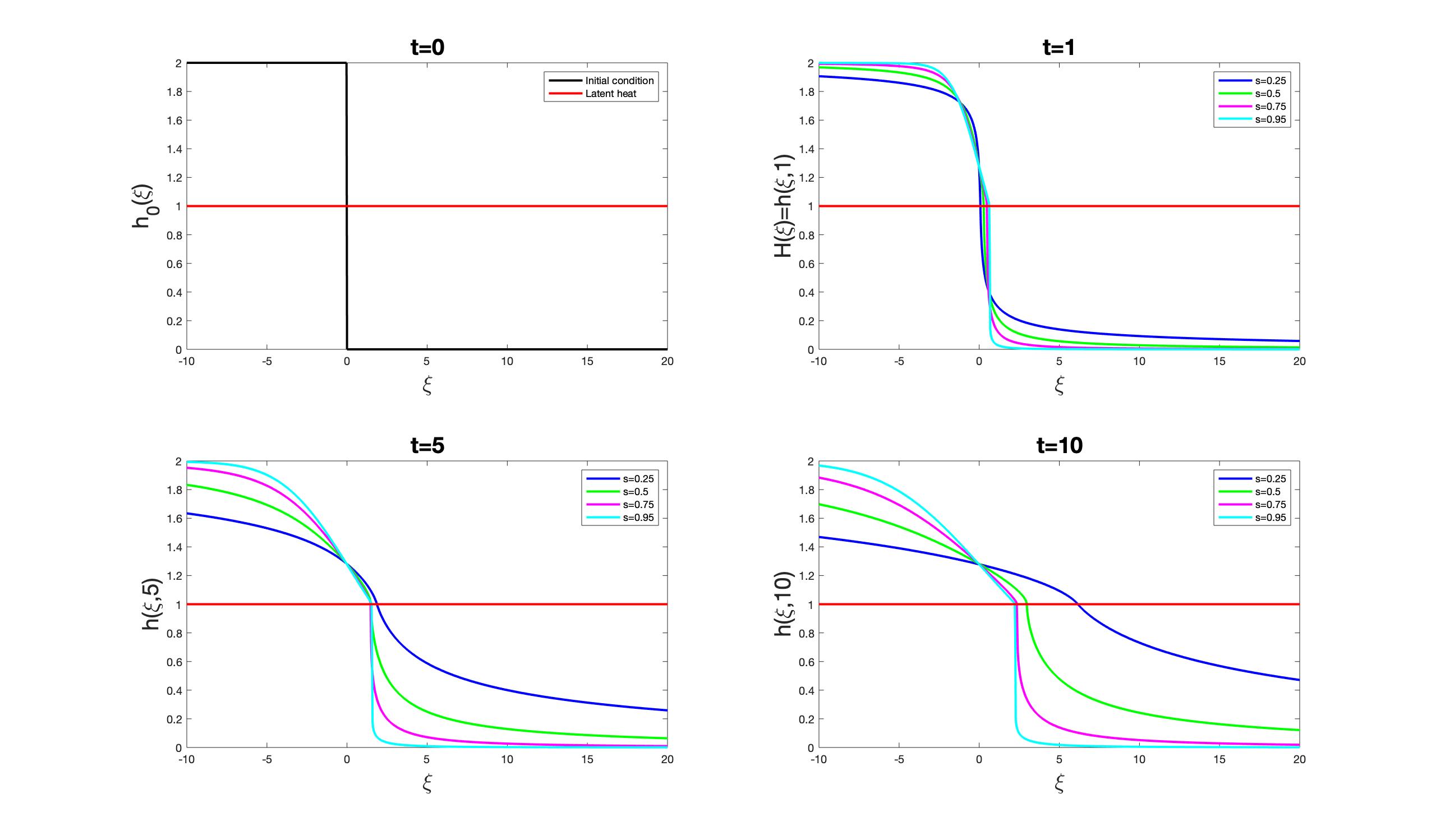}
\vspace{-1cm}
\caption{Selfsimilar solutions of Theorem \ref{thm:SS-all} for $L=P_1=P_2=1$.}
\label{fig:SSSmovings}
\end{figure}

As it can be seen in Figure \ref{fig:SSSmovings}, the behaviour of the selfsimilar solution strongly depends on $s$. Note that for short times the interface points of the four solutions are very close, this indicates that the interface parameter $\xi_0$ is very similar in these cases. However, for large times, the propagation is bigger the smaller is $s$. This is because the selfsimilar solution has the form $h(x,t)=H(x t^{-1/(2s)})$, and thus, the interface point is at $x=\xi_0 t^{1/(2s)}$.

Note also that as $s\to1^-$ the selfsimilar solution is more abrupt in the ice region close to the interface point, which is coherent with the fact that in the classical Stefan problem, the solutions are known to be discontinuous.

\subsubsection{Instant and nonconnected water emerging solutions}\label{subsubsec:instantemerg}
Here we present an example of a solution for which the water region is a connected domain at time $t=0$, but another not connected water region instantaneously emerges due to the nonlocal character of our equation. This phenomena is not present in the local Stefan problem but it has been shown for other nonlocal Stefan type problems like the one in \cite{ChS-G13}.

\begin{figure}[h!]
\includegraphics[width=1\textwidth]{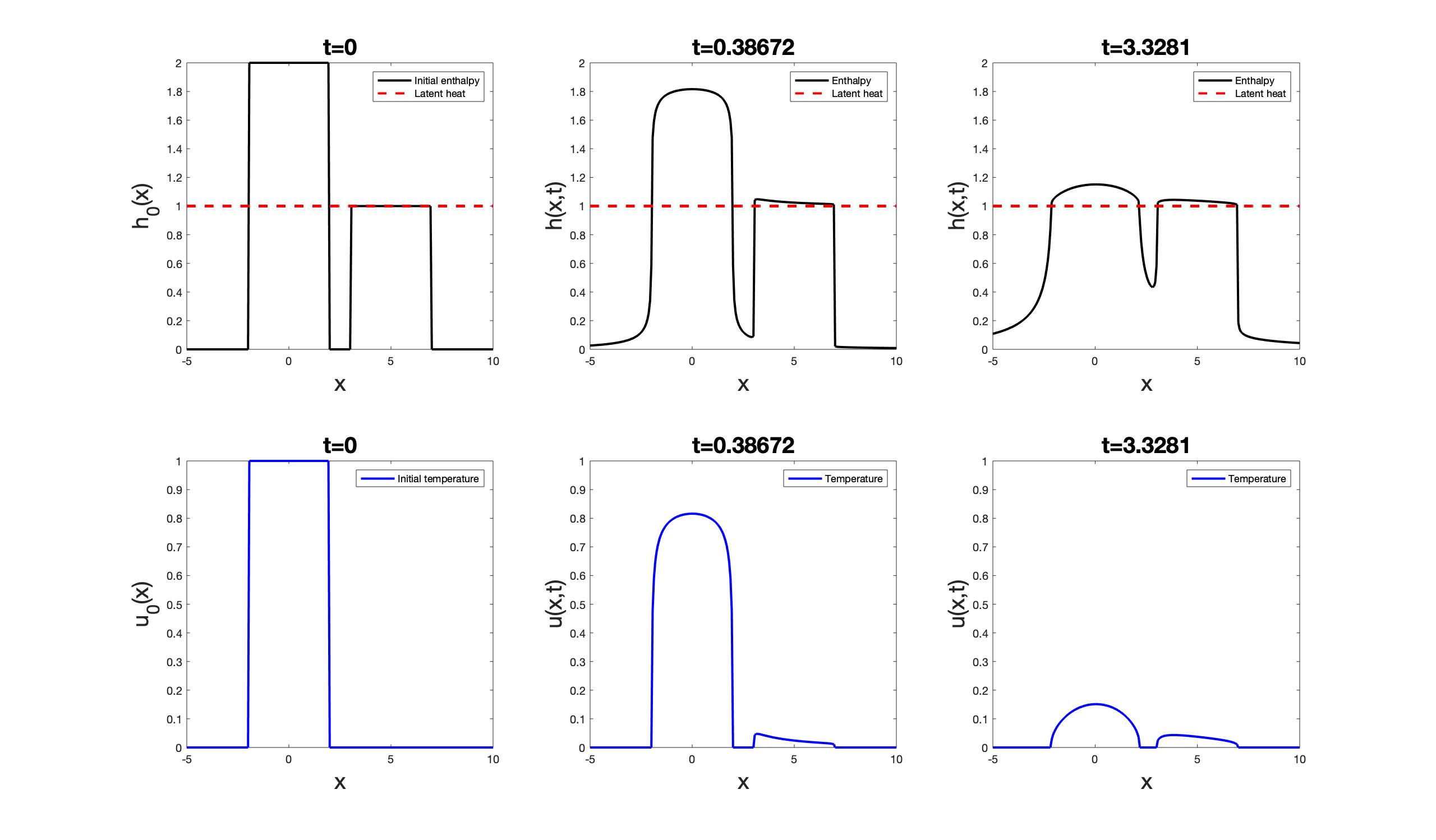}
\vspace{-1cm}
\caption{Instant water emerging solution for $s=0.25$ and $L=1$.}
\label{fig:instantemergin}
\end{figure}

Note that in Figure \ref{fig:instantemergin}, initially the water region is located in $B_2(0)$ while for any positive time another water region is also present in $B_2(5)$.

Some other interesting phenomena can also be observed in Figure \ref{fig:instantemergin} such as the infinite speed of propagation of the enthalpy predicted by Theorem \ref{thm:infinitespeedh}, the finite speed of propagation and the estimate on the maximum support of the temperature $u$ given in Theorem \ref{coro:NFiniteSpeed2} or the preservation of positivity regions of the temperature $u$ of Theorem \ref{thm:conspositivityu}.

\subsubsection{Noninstant and nonconnected water emerging regions}
The phenomenon of instant emerging regions of Section \ref{subsubsec:instantemerg} happen because the initial enthalpy was equal to the latent heat in the upcoming water emerging region. However, if we avoid such a situation, the emerging water region phenomenon cannot occur instantaneously, since we can compare with one of our selfsimilar solutions with finite speed of propagation of the temperature. Thus, the emerging region show up, but only after some time (see Figure \ref{fig:noninstantemergin}).

\begin{figure}[h!]
\includegraphics[width=\textwidth]{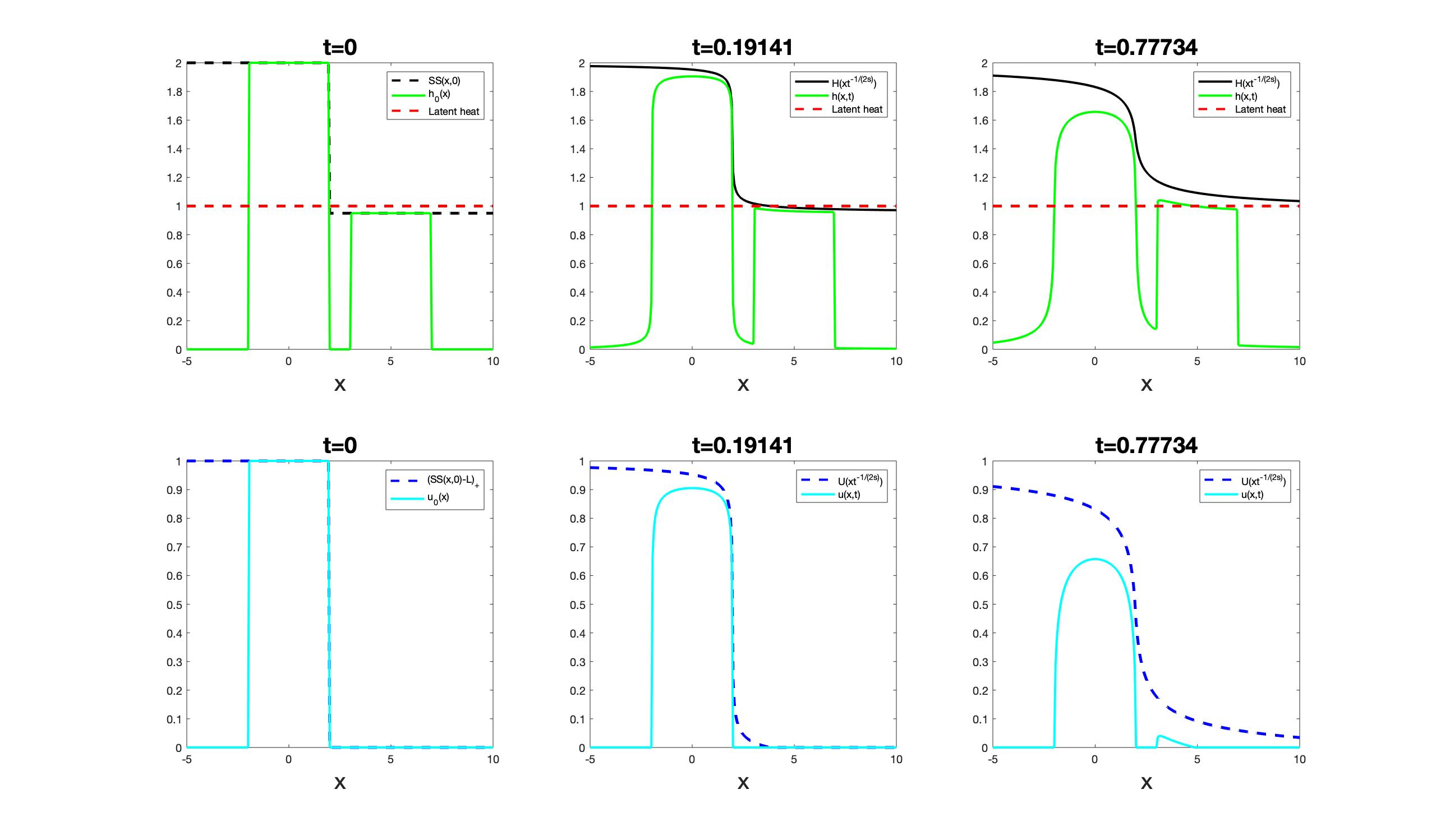}
\vspace{-1cm}
\caption{Noninstant water emerging solution for $s=0.25$ and $L=1$ compared with a selfsimilar solution of Theorem \ref{thm:SS-all}  with $P_1=1$ $P_2=0.05$.}
\label{fig:noninstantemergin}
\end{figure}

%%%%%%%%%%%%%%%%%%%%%%%%%%%%%%%%%%%%%%%%%%%%%%%%%%%%%%%%%%%NEW SECTION%%%%%%%%%%%%%%%%%
%%%%%%%%%%%%%%%%%%%%%%%%%%%%%%%%%%%%%%%%%%

\section{Comments and open problems}\label{sec:openprob}

$\bullet$ Some of the theory developed in this work also applies to the Two-Phase Fractional Stefan Problem, namely, problem \eqref{P1} with the nonlinear relation $\Phi$ replaced by
\begin{equation*}
\Phi_2(h):= k_1 \max\{h-L,0\} + k_2 \min\{h,0\}
\end{equation*}
with $L,k_1,k_2>0$. Much can be derived from the present paper for this two-phase problem. Thus, since $\Phi_2$ is  nondecreasing and Lipschitz, a priori properties, uniqueness and existence, and continuity of $u=\Phi_2(h)$, follow directly from the results of Appendix \ref{sec:WellPosednessInLInf}, \cite{GrMuPu19} and \cite{AtCa10} respectively. The theory of numerical schemes of Section \ref{sec:Num} also applies with no changes. The two-phase equation preserves the same scaling properties, thus we have existence of nonincreasing selfsimilar solutions with a profile $H$ satisfying a profile equation analogue to \eqref{P2SS}. However, fine properties of the selfsimilar solutions such us the continuity of $H$, strict monotonicity on the interface region,  or the trajectory of the free boundary are not straightforward consequences of our theory, and depend strongly  on the parameters of $\Phi_2$. Consequently, propagation properties in the spirit of the results developed in Sections \ref{sec.fpp} and \ref{sec.conspos} do not follow trivially. These matters are the object of a further study in \cite{DTEnVa20}.

$\bullet$ Next, we make some  comments on the comparison with the mid-range interaction model
\begin{equation}\label{eq:STmidrange}
\dell_t h= J*u- u, \quad u=(h-1)_+,
\end{equation}
studied in \cite{BrChQu12}. There are a number of differences with our results for the fractional one-phase Stefan model, like:

(i)  Their temperature $u$ has a waiting time, while for model \eqref{P1} a counterexample of this property is given by the selfsimilar solutions of Theorem \ref{thm:SS-all} and Corollary \ref{cor:NSS-all}.

(ii) In \eqref{eq:STmidrange} the enthalpy $h$ has finite speed of propagation, while in \eqref{P1} it has infinite speed of propagation, as shown in Theorem \ref{thm:infinitespeedh}.

(iii) In \eqref{eq:STmidrange} there is no regularizing effect, while in \eqref{P1} the selfsimilar solution is an example of a continuous enthalpy solution for all positive times with a discontinuous initial datum.

From a more theoretical and technical point of view, there are also many differences:

(iv) The fact that $J*u- u$ is a zero-order integro-differential operator allows them to work with classical solutions, and well-posedness is more or less straightforward. For us, we need to work with a more general concept of solution, which introduces several nontrivial technicalities such as uniqueness of a good class of solutions.

(v) The absence of scaling in \eqref{eq:STmidrange} does not allow for selfsimilar solutions, which are one of the main features of our model.

Of course, there are also some similarities between \eqref{P1} and  \eqref{eq:STmidrange} such as finite speed of propagation of the temperature, or the creation of nonconnected water regions.

We devote the remainder of this section to comment on open problems which are quite relevant to our work, but do not fit in the scope of this paper.

$\bullet$ As shown in Theorem \ref{thm:SS-all} and Corollary \ref{cor:NSS-all}, discontinuous initial data can produce enthalpy solutions which are continuous for all positive times. We wonder whether this phenomenon can be generalized: Is any solution corresponding to some $L^\infty$ initial data continuous (at least close the interface region $\{h=L\}$)? This continuity property is in general not true in the classical Stefan problem, nor in the mid-range interaction Stefan problem of \cite{BrChQu12}.

$\bullet$ We present our work in the context of bounded very weak solutions. It will be interesting to explore if these solutions are in fact more regular, such as weak energy solutions.  On this subject, we refer to Corollary 1.5 in \cite{GrMuPu19}, where it is shown that very weak solutions have at least a local energy in $L^2_{\textup{loc}}([0,\infty); H^s(\R^N))$.

$\bullet$ Here we prove that the selfsimilar profile $H$ is $C^\infty$ in the ice region and $C^{1,\alpha}$ in the water region. Is the regularity in the water region optimal, or can it be improved? Similar questions can be asked for general solutions.

$\bullet$ We note that \eqref{P1} can be written in the form of a conservation law involving nonlocal gradients:
\begin{equation}\label{eq:conlaw}
\partial_th= \nabla\cdot(h {\bf v}) \quad \textup{with} \quad {\bf v}=\frac{\nabla(-\Delta)^{s-1}u}{h}.
\end{equation}
Equations with the structure of \eqref{eq:conlaw} have also been studied in the recent literature. We refer to \cite{CaVa11, BiImKa15, StTeVa16, StTeVa19} for works regarding well-posedness and properties of propagation.

$\bullet$  Can we derive from \eqref{eq:conlaw}  an equation for the free boundary $y(t)$? In the 1-D  local Stefan Problem, the free boundary equation reads
\begin{equation*}
\partial_t  y(t) =-L^{-1}\partial_x u(y(t),t),
\end{equation*}
with $y(0)=0$. The nonlocal counterpart we suggest  is given by
\[
\partial_t  y(t)=L^{-1}\partial_x(-\partial_{xx})^{s-1} u(y(t),t).
\]
In the terminology of \cite{BiImKa15}, $\nabla^{2s-1}:=\nabla (-\Delta)^{s-1}$ is called the nonlocal gradient.

$\bullet$ It will also be interesting to study the large time behaviour of general solutions with compactly supported initial data in the spirit of \cite{BrChQu12}, where they are proved to converge to a mesa type profile.

$\bullet$ Much work is needed to understand the behaviour of the free boundaries of general solutions and their limits as $s$ tends to the limits $0$ and $1$.

\appendix

\section{Very weak solutions of the generalized porous medium equation}\label{sec:WellPosednessInLInf}
\label{sec.app1}
The goal of this section is to obtain properties of a class of bounded very weak solutions for initial data and right-hand sides in $L^\infty$. We consider the following more general equation:
\begin{equation}\label{eq:genproblem}
\begin{cases}
\dell_tv(x,t)-\Operator[\Phi(v)](x,t)=f(x,t) \qquad\qquad&\text{in}\qquad Q_T:=\R^N\times(0,T)\\
v(\cdot,0)=v_0 \qquad\qquad&\text{on}\qquad \R^N,
\end{cases}
\end{equation}
where $v$ is the solution, $\Phi$ is a merely
continuous and nondecreasing function, $f=f(x,t)$ some right-hand side, and $T>0$. The nonpositive, symmetric operator $\Operator$ is given as  $\Operator:=L^\sigma+\Levymu$ with a local part $  L^\sigma[\psi](x):=\text{tr}\big(\sigma\sigma^TD^2\psi(x)\big)$ and a nonlocal one,
\[
\Levy ^\mu [\psi](x):=\textup{P.V.}\int_{\R^N\setminus\{0\} } \big(\psi(x+z)-\psi(x)\big) \dd\mu(z),
\]
where  $\psi \in C_\textup{c}^2(\R^N)$, $\sigma=(\sigma_1,....,\sigma_P)\in\R^{N\times P}$ for $P\in \N$
and $\sigma_i\in \R^N$, $D^2$ is the Hessian, and $\mu$ is a nonnegative
symmetric Radon measure. In fact, we assume that
\begin{align}
&\Phi:\R\to\R\text{ is nondecreasing and continuous; and}
\tag{$\textup{A}_\Phi$}&
\label{phias}\\
&\label{muas}\tag{$\textup{A}_{\mu}$} \mu \text{ is a nonnegative symmetric Radon measure on
}\R^N\setminus\{0\}
\text{ satisfying}
\nonumber\\
&\qquad\qquad \int_{|z|\leq1}|z|^2\dd \mu(z)+\int_{|z|>1}1\dd
\mu(z)<\infty\nonumber.
\end{align}

The general theory developed in \cite{DTEnJa17a, DTEnJa17b, DTEnJa18b, DTEnJa19} mainly regards $L^1\cap L^\infty$ very weak solutions of  \eqref{eq:genproblem}. However, for the purpose of this paper, we need to develop a pure $L^\infty$ theory.

Very weak solutions of \eqref{eq:genproblem}  are defined just as in Definition \ref{def:distSolfrac} replacing $(-\Delta)^s$ by $\Operator$ and letting $\Phi(v)$ be as in \eqref{phias} rather than just $(v-L)_+$. From the papers \cite{DTEnJa17a, DTEnJa17b, DTEnJa19}, we have existence, uniqueness, and properties of very weak solutions with $L^1\cap L^\infty$ data. From there, we now prove a result regarding pure $L^\infty$ solutions.

\begin{theorem}[Existence and properties]\label{thm:Linfty-dist}
Assume \eqref{phias}, \eqref{muas}, $ v_0 \in L^\infty(\R^N)$, and $ f\in L^1((0,T):  L^\infty(\R^N))$. Then
\begin{enumerate}[{\rm (i)}]
\item\label{thmlinf-itema} There exists a very weak solution $v\in  L^\infty(Q_T)$ of \eqref{eq:genproblem}.
\item\label{thmlinf-itemb} Let $v,\vv$ be two very weak solutions of \eqref{eq:genproblem} constructed in \eqref{thmlinf-itema}  with datum $(v_0,f), (\vv_0,\hat{f})$  respectively. We have that
\begin{enumerate}[{\rm (a)}]
\item\label{thmlinf-item1}  $\|v(\cdot,t)\|_{L^\infty(\R^N)}\leq \|v_0\|_{L^\infty(\R^N)}+ \|f\|_{L^1((0,t): L^\infty(\R^N))}$ for a.e. $t\in(0,T)$.
\item\label{thmlinf-item2} If $v_0\leq \vv_0$ and $f\leq \hat{f}$ a.e., then $v \leq\vv$ a.e.
\item\label{thmlinf-item3}  Assume, in addition, that the data satisfies $(v_0-\vv_0)^+\in L^1(\R^N)$ and $(f-\hat{f})^+\in L^1(Q_T)$. Then, for a.e. $t\in(0,T)$
\[
\int_{\R^N}(v(x,t)-\vv(x,t))^+\dd x\leq \int_{\R^N}(v_0(x)-\vv_0(x))^+\dd x+ \int_0^t\int_{\R^N}(f(x,s)-\hat{f}(x,s))^+\dd x\dd s.
\]
\item\label{thmlinf-item4} If $\|v_0(\cdot+\xi)-v_0\|_{L^1(\R^N)}, \|f(\cdot+\xi,\cdot)-f\|_{L^1(Q_T)}\to 0$ as $|\xi|\to 0^+$ then $v\in C([0,T];L^1_{\textup{loc}}(\R^N))$. Moreover, for every $t,s\in[0,T]$ and compact set $K\subset\R^N$,
$$
\|v(\cdot,t)-v(\cdot,s)\|_{L^1(K)}\leq \Lambda_K(|t-s|) +|K| \int_s^t \|f(\cdot,t)\|_{L^\infty(\R^N)}.
$$
where $\Lambda_K$ is a modulus of continuity depending on $K$, $\|v_0(\cdot+\xi)-v_0\|_{L^1(\R^N)}$ and $\|f(\cdot+\xi,\cdot)-f\|_{L^1(Q_T)}$.
\end{enumerate}
\end{enumerate}
\end{theorem}

\begin{remark}\rm
 Uniqueness is unknown for general $L^\infty$ very weak solutions of \eqref{eq:genproblem}.  However, in \cite{DTEnJa17a, DTEnJa17b} it is proved under the extra assumption that $v-\hat{v}\in L^1(Q_T)$. This extra condition has been recently removed in \cite{GrMuPu19} when $\Phi$ is locally Lipschitz and $\Operator=(-\Delta)^s$, and this is enough for the purpose of the current paper. Unfortunately, the uniqueness technique of \cite{GrMuPu19} does not seem to be easily adaptable to the general context of  \eqref{eq:genproblem},  but this is not of concern here.
 \end{remark}

\begin{proof}[Proof of Theorem \ref{thm:Linfty-dist}]
The proof of Theorem \ref{thm:Linfty-dist} will consist in an approximation by $L^1\cap L^\infty$ very weak solutions. It is also possible to obtain uniqueness in the class of  $L^\infty$ solutions that are approximated by $L^1\cap L^\infty$  solutions. Since this result is not relevant for this paper, and will require a more delicate formulation, we omit it.

To prove existence we will approximate by $L^1\cap L^\infty$ solutions in a particular way; in fact, we follow \cite{AnBr19} (see also \cite{AmWi03}). Let $\rho,\eta>0$ and define
\[
v_0^{\rho,\eta}(x)= \max\{v_0(x),0\} \indik_{B_\rho}(x)+ \min\{v_0(x),0\} \indik_{B_\eta}(x),
\]
and
\[
f^{\rho,\eta}(x,t)= \max\{f(x,t),0\} \indik_{B_\rho}(x)+ \min\{f(x,t),0\} \indik_{B_\eta}(x).
\]
Clearly $v_0^{\rho,\eta}\in L^1(\R^N)\cap L^\infty (\R^N)$ and $f^{\rho,\eta}\in L^1((0,T):  L^1(\R^N)\cap L^\infty (\R^N))$. Then, by the results in \cite{DTEnJa17a, DTEnJa17b}, there exists a unique very weak solution $v^{\rho,\eta}\in L^1(Q_T)\cap L^\infty (Q_T)$ of \eqref{eq:genproblem} with datum $v_0^{\rho,\eta}$ and $f^{\rho,\eta}$. These solutions also satisfies properties \eqref{thmlinf-item1}--\eqref{thmlinf-item4}. In particular
\begin{equation*}%\label{eq:unifbounl1aprox}
\begin{split}
\|v^{\rho,\eta}(\cdot,t)\|_{L^\infty(\R^N)}&\leq \|v_0^{\rho,\eta}\|_{L^\infty(\R^N)}+ \|f^{\rho,\eta}\|_{L^1((0,t): L^\infty(\R^N))}\\
&\leq  \|v_0\|_{L^\infty(\R^N)}+ \|f\|_{L^1((0,t): L^\infty(\R^N))}=: M_{v_0,f}<+\infty.
\end{split}
\end{equation*}
Thus, $v^{\rho,\eta}(x,t)$ is uniformly bounded (form above and below) in $\rho$ and $\eta$. It is also clear that given any $\veps>0$ we have that
\[
v_0^{\rho,\eta}\leq v_0^{\rho+\veps,\eta}, \quad v_0^{\rho,\eta+\veps}\leq v_0^{\rho,\eta}, \quad f^{\rho,\eta}\leq f^{\rho+\veps,\eta}, \quad \textup{and} \quad  f^{\rho,\eta+\veps}\leq f^{\rho,\eta},
\]
so by comparison (as stated in \eqref{thmlinf-item2}) we get that  $v^{\rho,\eta}\leq v^{\rho+\veps,\eta}$ and $v^{\rho,\eta+\veps}\leq v^{\rho,\eta}$.

We have proved that the family $\{v^{\rho,\eta}(x,t)\}_{\rho>0}$ is uniformly nondecreasing (in $\rho$) for almost every $(x,t)\in Q_T$, and is also uniformly bounded. Therefore we have that there exists a measurable function $v^\eta\in L^\infty(Q_T)$ such that
\[
v^\eta=\lim_{\rho\to+\infty} v^{\rho,\eta} \quad \textup{for a.e.} \quad (x,t)\in Q_T.
\]
Now $\{v^\eta(x,t)\}_{\eta}$ is nonincreasing and bounded, and thus, there exists a measurable function $v\in L^\infty(Q_T)$ such that $v=\lim_{\eta\to+\infty} v^{\eta}$ for a.e. $(x,t)\in Q_T$. So we have that
\[
v=\lim_{\eta\to+\infty}\lim_{\rho\to+\infty} v^{\rho,\eta} \quad \textup{for a.e.} \quad (x,t)\in Q_T
\]
and $\|v\|_{L^\infty(Q_T)} \leq M_{v_0,f}$. We will show that $v$ is a very weak solution of \eqref{eq:genproblem}. Note that, for given $\psi\in C_\textup{c}^\infty(Q_T)$, $|v^{\rho,\eta}| |\partial_t\psi|\leq M_{v_0,f}|\partial_t\psi|\in L^1(Q_T)$ and $|\Phi(v^{\rho,\eta})| |\Operator[\psi]|\leq |\Phi(M_{v_0,f})||\Operator[\psi]|\in L^1(Q_T)$. Since $v^{\rho,\eta}\to v$, and then also $\Phi(v^{\rho,\eta})\to \Phi(v)$, pointwise a.e. in $Q_T$ as $R\to+\infty$, we can use the Lebesgue dominated convergence theorem to obtain
\[
\int_0^T \int_{\R^N} v^{\rho,\eta} \partial_{t}\psi \dd x\dd t \stackrel{\rho\to+\infty}{\longrightarrow}\int_0^T \int_{\R^N} v^\eta \partial_{t}\psi\dd x\dd t\stackrel{\eta\to+\infty}{\longrightarrow}\int_0^T \int_{\R^N} v \partial_{t}\psi \dd x\dd t
\]
and
\[
\begin{split}
\int_0^T \int_{\R^N} \Phi(v^{\rho,\eta}) \Operator[\psi] \dd x\dd t &\stackrel{\rho\to+\infty}{\longrightarrow}\cdots \stackrel{\eta\to+\infty}{\longrightarrow}\int_0^T\int_{\R^N} \Phi(v) \Operator[\psi] \dd x\dd t.
\end{split}
\]
Similarly, we pass to the limit in the terms involving $v_0$ and $f$, which shows that $v$ is a very weak solution of \eqref{eq:genproblem}. This proves the existence result given in \eqref{thmlinf-itema} together with the bound in \eqref{thmlinf-item1}. Part \eqref{thmlinf-item2} follows trivially from the pointwise convergence. A standard use of the monotone convergence theorem shows that we can pass to the limit in the estimate in \eqref{thmlinf-item3}. Finally \eqref{thmlinf-item4} follows easily using the dominated convergence theorem.
\end{proof}

 We show now how the result in \cite{AtCa10} is applied to get continuity of $\Phi(h)$ in the context of very weak solutions of \eqref{P1}-\eqref{P1-Init}.

\begin{proof}[Proof of Theorem \ref{thmfrac-cont}]
The authors  of \cite{AtCa10} work with the variable $u=\Phi(h)$ and approximate problem \eqref{P1} by regularizing $\Phi$,  or more precisely $\beta=\Phi^{-1}$, by a  sequence of nondegenerate  increasing nonlinearities $\beta_\veps$.  By the results of \cite{DPQuRoVa17}, the regularized problem admits a unique weak energy solution $u_\veps$ when the data belong to $L^1(\R^N)\cap L^\infty(\R^N)$, and that solution is  uniformly bounded, and also smooth if $\beta_\veps$ is smooth. Recall that here $u_0\ge 0$ and $u_\veps\ge 0.$   Sections 2 and 3  in \cite{AtCa10} prove that these  uniformly bounded solutions $u_\veps$  are equi-continuous for $t\ge \tau>0$ with a given modulus of continuity that does not depend on $\veps$, only on $\beta$  and the norm of $h_0$ in $L^\infty$. On one hand, we can thus extract a subsequence of $\{u_\veps\}_{\veps}$ converging locally uniformly to a function $u\in L^\infty(Q_T)$  which is continuous with the same modulus for all $t\ge \tau>0$. On the other hand, the sequence $h_\veps:=\beta_{\veps}(u_\veps)=\Phi^{-1}_\veps (u_\veps)$  is uniformly bounded in $Q_T$, hence it converges after  extraction of  yet another subsequence (locally weakly in every $L^p$, $1<p<\infty$) to some $h\in L^\infty(Q_T)$. We conclude that $u=\Phi(h)$  by using standard properties of limits for monotone nonlinearities. By taking limits it is clear that the pair $(h,u)$ produces a very weak solution of \eqref{P1}--\eqref{P1-Init} in the sense of Definition \ref{def:distSolfrac}. For data in  $L^\infty(\R^N)$ but not in $L^1(\R^N)$, we  proceed by monotone approximation of the initial data as in the proof of Theorem \ref{thm:bdddistsol}.
Uniqueness of very weak solutions (Theorem \ref{thm:bdddistsol})  ensures that \cite{AtCa10}'s result applies to our class of solutions.
   Finally, the last statement is easy.
\end{proof}

\section{Auxiliary results on the fractional Laplacian}
\label{sec.aux}

Straightforward computations give

\begin{lemma}[Homogeneity]\label{lem:FLHomogeneity}
Assume $s\in(0,1)$, $a>0$, and let $\psi\in C_\textup{b}^2(\R^N)$. Then
$$
\FL[\psi(a\cdot)](x)=a^{2s}\FL \psi(ax).
$$
\end{lemma}

We need the following technical result whose proof is standard and we omit it.

\begin{lemma}\label{lem:Uice}
Assume $s\in(0,1)$, and let $\psi\in L^\infty(\R)$ such that $\psi>0$ in $(-\infty,x_0)$ and $\psi=0$ for all $x\in[x_0,+\infty)$ for some $x_0\in \R$. Then:
\begin{enumerate}[{\rm (a)}]
\item $\FL\psi(x)<0$ for all $x\in (x_0,+\infty)$.
\item $\FL \psi \in C^\infty((x_0,\infty))$.
\item $\FL \psi$ is strictly increasing for all $x\in(x_0,+\infty)$.
\item $\FL \psi (x)\asymp 1/|x-x_0|^{2s}$ for all $x\gg x_0$.
\end{enumerate}
\end{lemma}
We also want to use some standard results on cut-off functions: Let $\eta_R(x):=\eta(x/R)$ with $R>0$ where $0\leq\eta \in C^\infty(\R^N)$ is a radially symmetric non-increasing in $|x|$ function such that $\eta(x)=1$ if $|x|\leq 1/2$ and $\eta(x)=0$ if $|x|>1$.
We immediately have that $\nabla\eta_R(x)=\frac{1}{R}\nabla\eta(\frac{x}{R})$, $\eta_R$ is compactly supported, $|\eta_R(x)|\leq 1$, and $\eta_R\nearrow 1$ as $R\to+\infty$. Moreover:
\begin{lemma}\label{lem:CutOff}
Assume $s\in(0,1)$. Then:
\begin{enumerate}[{\rm (a)}]
\item\label{item1:lem:CutOff} $|\FL\eta_R(x)|=R^{-2s}|\FL\eta(x/R)|\leq CR^{-2s}(1+|x/R|)^{-N-2s}$.
\item\label{item2:lem:CutOff} $\FL\eta_R\in L^1(\R^N)$.
\end{enumerate}
\end{lemma}
\begin{proof}
It follows from the fact that $|\FL\eta(x)|\leq\frac{C}{(1+|x|)^{N+2s}}$ which can be found in e.g. \cite{DPQuRoVa12}.
\end{proof}

The following technical result allows us to inherit some of the properties of the solution in $\R$ to solutions in $\R^N$. Again, the proof is straightforward (for example using the semigroup representation of $\FL$) and we omit it.
\begin{lemma}\label{lem:secSol}
Assume $s\in(0,1)$, and let $f\in C_\textup{c}^\infty(\R,\R)$ and let $F\in C_\textup{c}^\infty(\R, \R^N)$ be defined by $F(x,x'):=f(x)$ for all $x'\in \R^{N-1}$. Then
\[
(-\Delta_N)^{s} F(x,x')=(-\Delta_1)^{s} f(x).
\]
\end{lemma}

\subsection*{Acknowledgments}
The research of F.~del~Teso has received funding from grants BCAM Severo Ochoa accreditation SEV-2017-0718 and PGC2018-094522-B-I00 from the MICINN of the Spanish Government; J.~Endal from the Research Council of Norway under the Toppforsk (research excellence) grant agreement no. 250070 ``Waves and Nonlinear Phenomena (WaNP)'', from the European Union’s Horizon 2020 research and innovation programme under the Marie Sk\l odowska-Curie grant agreement no. 839749 ``Novel techniques for quantitative behaviour of convection-diffusion equations (techFRONT)'', and from the Research Council of Norway under the MSCA-TOPP-UT grant agreement no. 312021; and J.~L.~V\'azquez from grant PGC2018-098440-B-I00 from the MICINN of the Spanish Government. J.~L.~V\'azquez is also an Honorary Professor at Univ. Complutense de Madrid. We are grateful to A.~Figalli, E.~R.~Jakobsen, and X.~Ros-Oton for fruitful discussions.

%%%%%%%%%%%%%%%%%%%%%%%%%%%%%%%%%%%%%%%%%%%%%%%%%%%%%%%%%%%NEW SECTION%%%%%%%%%%%%%%%%%
%%%%%%%%%%%%%%%%%%%%%%%%%%%%%%%%%%%%%%%%%%

\lhead{\emph{References}}
\bibliographystyle{abbrv}

\end{document}